\documentclass[12pt]{amsart}

\usepackage[]{graphicx}
\usepackage{amssymb}
\usepackage{amsmath}
\usepackage{amsthm}
\usepackage{mathrsfs}
\usepackage{comment}
\usepackage{lscape}
\usepackage{overpic}

\theoremstyle{plain}
\newtheorem{thm}{Theorem}[section]

\newtheorem{prop}[thm]{Proposition}

\theoremstyle{definition}
\newtheorem{defn}[thm]{Definition}
\newtheorem{exam}[thm]{Example}
\newtheorem{rem}[thm]{Remark}

\title{INVARIANTS OF HANDLEBODY-KNOTS VIA YOKOTA'S INVARIANTS}
\author{Atsuhiko Mizusawa}
\author{ Jun Murakami}
\address{Department of Mathematics, Fundamental Science and Engineering, Waseda University, 3-4-1Okubo, Shinjuku-ku, Tokyo 169-8555, Japan}
\email[Jun Murakami]{murakami@waseda.jp}
\email[Atsuhiko Mizusawa]{a\_mizusawa@aoni.waseda.jp}
\date{\today}

\begin{document}

\begin{abstract}
 We construct quantum $\mathcal{U}_q(\mathfrak{sl}_{\,2})$ type invariants for handlebody-knots in the 3-sphere $S^3$. A handlebody-knot is an embedding of a handlebody in a 3-manifold. These invariants are linear sums of Yokota's invariants for colored spatial graphs which are defined by using the Kauffman bracket. We give a table of calculations of our invariants for genus 2 handlebody-knots up to six crossings. We also show our invariants are identified with special cases of the Witten-Reshetikhin-Turaev invariants.
\end{abstract}


\maketitle

\section{Introduction} \label{sec1}

\par
Throughout this paper we work in the piecewise linear category. A spatial graph is an embedding of a finite graph into a 3-manifold. Suzuki \cite{Su} introduced the notion of the neighborhood equivalence for spatial graphs. Two spatial graphs are neighborhood equivalent if they have the same regular neighborhood up to isotopy of the 3-manifold. Ishii \cite{Ishi} reformulated this notion as a handlebody-knot which is an embedding of a handlebody into a 3-manifold. If the genus of the handlebody is 1, the handlebody-knot is regarded as an ordinary knot. 
\par
A handlebody-knot in the 3-sphere $S^3$ is represented by a diagram of a spatial trivalent graph whose regular neighborhood is isotopic to the handlebody-knot. Ishii \cite{Ishi} showed that two diagrams representing the same handlebody-knot are transformed to  each other by a sequence of six local moves called ``Reidemeister moves'' for handlebody-knots. Five of them are the Reidemeister moves for spatial trivalent graphs. Therefore, we can make invariants of handlebody-knots from invariants of spatial trivalent graphs by modifying them to satisfy the new Reidemeister move.
\par
In this paper we construct quantum $\mathcal{U}_q(\mathfrak{sl}_2)$ type invariants for handlebody-knots in $S^3$ via Yokota's invariants \cite{YO} which are quantum $\mathcal{U}_q(\mathfrak{sl}_2)$ type invariants for spatial graphs in $S^3$. Yokota's invariants are generalized to the relativistic invariants \cite{BGM} for spatial graphs in any closed oriented 3-manifold. The arguments in this paper hold for the relativistic invariants and our invariants are generalized for any closed oriented 3-manifold. In this paper we focus on Yokota's invariants since it is essential.
\par
There are various invariants for handlebody-knots. The Alexander ideals are known as invariants for neighborhood equivalence classes of spatial graphs \cite{Kin} that are handlebody-knots. Invariants using the Fox coloring \cite{Ishi} and quandle cocycle invariants \cite{Ishi, IshiIwa} were also defined for handlebody-knots. However quantum type invariants for handlebody-knots have not been defined yet
\footnote{A construction of quantum invariants for handlebody-knots via Yetter-Drinfeld module was announced by \cite{IshiMas}.} 
.
\par
We also show that our invariants are identified with special cases of the Witten-Reshetikhin-Turaev (WRT) invariants \cite{RT} for closed oriented 3-manifolds. 
\par
In this paper, we first introduce a handlebody-knot and its Reidemeister moves in Sec. 2. We define quantum $\mathcal{U}_q(\mathfrak{sl}_2)$ type invariants for the handlebody-knots via Yokota's invariants in Sec. 3. In Sec. 4, a table of calculations of our invariants is shown, which tells us that our invariants are non-trivial. We also show some properties of the invariants. In Sec. 5, it is shown that our invariants are identified with special cases of the WRT invariants.
 
\section{Handlebody-Knot and Handlebody-Link} \label{sec2}
\par
 A \textit{handlebody-knot} is an embedding of a handlebody into a 3-manifold and  a \textit{handlebody-link} is a disjoint union of embeddings of handlebodies into a 3-manifold (Fig.~\ref{fig1}). Throughout this paper, the notion of a handlebody-link includes a handlebody-knot. Two handlebody-links are \textit{equivalent} if there is an isotopy of the 3-manifold which transforms one handlebody-link to another. 
\begin{figure} [h] 
  $$ \raisebox{-13 pt}{\includegraphics[bb=  0 0 100 38, width=95 pt]{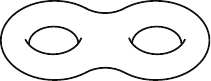}} \,\,\, \mbox{ \Large{$\hookrightarrow$}} \,\,\, \raisebox{30 pt}{\begin{overpic}[bb=0 0 54 54, height=20 pt]{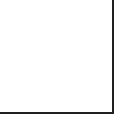}\put(5,30){\mbox{$M^3$}}\end{overpic}} \,\, \raisebox{-18 pt}{\includegraphics[bb= 0 0 102 87, height=65 pt]{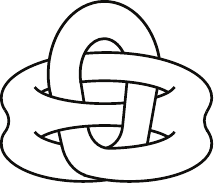}}$$
  \caption{An embedding of a genus 2 handlebody into a 3-manifold $M^3$.} \label{fig1}
\end{figure}
\par
We have a map from spatial graphs to handlebody-links that takes the regular neighborhoods of the spatial graphs (Fig.~\ref{fig2}). A handlebody-link in $S^3$ is called \textit{trivial} if it is equivalent to handlebodies standardly embedded in $S^3$. For handlebody-links in $S^3$, through the above map, we can handle them by diagrams of spatial graphs. However, there are infinitely many diagrams that represent the same handlebody-link (Fig.~\ref{fig3}). In \cite{Ishi}, Ishii showed that diagrams of spatial trivalent graphs have a one to one correspondence to handlebody-links up to local moves called Reidemeister moves (for handlebody-links). This means 
$$\{\mbox{handlebody-links}\} \!=\! \{\mbox{diagrams\! of spatial\! trivalent\! graphs}\}/\{ \rm{Reidemeister \,\, moves}\}.$$
Here spatial trivalent graphs may have circle components that correspond to genus 1 components of handlebody-links. These circles do not have any vertices or edges. 
The Reidemeister moves for handlebody-links are shown in Fig.~\ref{fig4}. The moves RI-RV are the Reidemeister moves for spatial trivalent graphs. 
\begin{figure}[ht]
$$  \raisebox{-17 pt}{\includegraphics[bb= 0 0 87 63, height=60 pt]{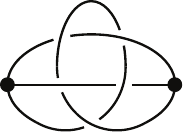}} \,\,\,\, \mbox{\Large{$\rightarrow$}} \,\,\,\, \raisebox{-18 pt}{\includegraphics[bb= 0 0 102 87, height=63 pt]{handlebody-knot01.pdf}}$$
\caption{A spatial graph represents a handlebody-link.} \label{fig2}
\end{figure}
\begin{figure}[ht]
$$  \raisebox{-11 pt}{\includegraphics[bb= 0 0 104 41, width=75 pt]{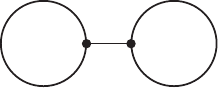}} \,\, \mbox{\Large{$\rightarrow$}} \,\,  \raisebox{-13 pt}{\includegraphics[bb=  0 0 100 38, width=90 pt]{handlebody01.pdf}}
\,\, \mbox{\Large{$\approx$}} \,\, \raisebox{-17 pt}{\includegraphics[bb= 0 0 108 71, width=65 pt]{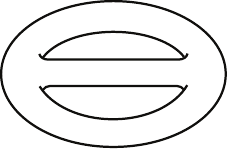}} \,\, \mbox{\Large{$\leftarrow$}} \,\, \raisebox{-13 pt}{\includegraphics[bb= 0 0 113 71, width=55 pt]{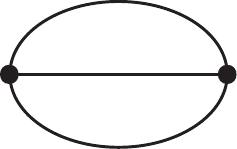}}$$
\caption{Two different spatial graphs that represent the trivial genus 2 handlebody-knot.} \label{fig3}
\end{figure}
\begin{figure}[ht]
\begin{centering}
 \includegraphics[bb= 0 0 93 53, height=55 pt]{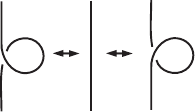}\hspace{1.0cm} \includegraphics[bb= 0 0 80 65, height=55 pt]{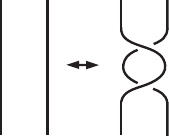} \hspace{1.0cm} \includegraphics[bb= 0 0 104 41, width=100 pt, height=55 pt]{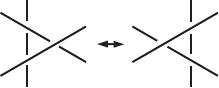}\\
 RI \hspace{3.35cm} RII \hspace{3.35cm} RIII\\
 \vspace{0.1cm}
  \includegraphics[bb= 0 0 109 32, height=50 pt]{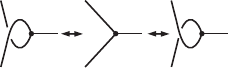} \hspace{0.5cm} \includegraphics[bb= 0 0 109 40, height=50 pt]{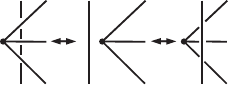}\\
  \hspace{0.8cm}RIV \hspace{5.2cm} RV \\
 \vspace{0.1cm}
\includegraphics[bb= 0 0 93 35, height=50 pt]{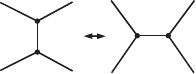} \\
   RVI\\
\end{centering}
\caption{Reidemeister moves for handlebody-links.} \label{fig4}
\end{figure}

\section{Quantum $\mathcal{U}_q(\mathfrak{sl}_2)$ Type Invariants for Handlebody-Links} \label{sec3}
\par
 In this section we define quantum $\mathcal{U}_q(\mathfrak{sl}_2)$ type invariants for handlebody-links in the 3-sphere $S^3$.

\subsection{Skein space and Jones-Wenzl idempotent}
\par
 To define quantum $\mathcal{U}_q(\mathfrak{sl}_2)$ type invariants for handlebody-links, we introduce the skein space and its special elements the Jones-Wenzl idempotents.

\begin{defn}[Skein space]
 Let $F$ be a connected and oriented 2-manifold (possibly with boundaries). A \textit{link diagram} in $F$ consists of finitely many closed curves and arcs whose endpoints are at the boundaries.  These curves may have finitely many transverse crossings that have information which of the two arcs is upper or lower. Two link diagrams are regarded as the same if there is an isotopy of $F$ that changes one link diagram to another fixing $\partial F$.
  \par
Let $A$ be a fixed value in $\mathbb{C}\setminus \{0\}$. A \textit{skein space} $\mathcal{S}(F)$ of $F$ is the vector space of formal linear sums over $\mathbb{C}$ of link diagrams in $F$ subject to the following relations, 
\begin{eqnarray*}
 & &\raisebox{-17 pt}{\includegraphics[bb= 0 0 73 73, width=40 pt, height=40 pt]{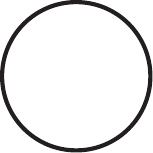}}  \raisebox{-2 pt}{\mbox{\Large{ $\sqcup$ $D$}}} 
 \, =-(A^2+A^{-2})  \raisebox{-2 pt}{\mbox{\Large{ $D$}}} \, ,\\
 & & \raisebox{-17 pt}{\includegraphics[bb= 0 0 72 71, width=40 pt, height=40 pt]{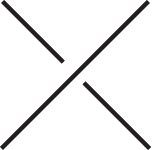}}
 \, =A \, \raisebox{-17 pt}{\includegraphics[bb= 0 0 72 71, width=40 pt, height=40 pt]{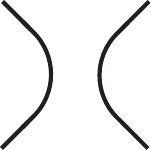}} 
  + A^{-1} \, \raisebox{-17 pt}{\includegraphics[bb= 0 0 71 72, width=40 pt, height=40 pt]{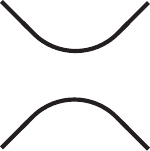}} \,\,,
 \end{eqnarray*}
where the left-hand side of the first relation is a disjoint union of a trivial circle and a link diagram $D$ and in the second relation the figures represent parts of link diagrams and the complements of them are the same diagrams.
\end{defn}
 \par
 If $F=S^2$, since $\partial F = \emptyset$, all curves in $\mathcal{S}(F)$ are closed. By the relations, the closed curves become scalar multiplied blank diagrams. Thus $\mathcal{S}(S^2)$ is identified with $\mathbb{C}$. For each link diagram in $S^2$, we represent the corresponding complex value as the diagram inside the \textit{Kauffman bracket} $\left< \,\cdot\, \right>$. 

   Note that the value of the Kauffman bracket is an invariant of framed links (i.e. it does not change under the RII and RIII moves).
 \par
 We fix $3\leq r \in \mathbb{N}$ and $A = {\rm e}^{2\pi i/4r}$ for later use. Let $D^2_{n}$ be a 2-disc that has $2n$ points in the boundary. We regard it as a rectangle whose left edge and right edge have just $n$ points respectively.

\begin{defn}[Jones-Wenzl idempotent]
 Following \cite{Lic}, for $0\leq n \leq r-1$, we define a special element $JW_n \in \mathcal{S}(D^2_{n})$ illustrated as a white box in diagrams. $JW_n$ is defined recursively as follows, 
$$
 \raisebox{-17 pt}{\includegraphics[bb= 0 0 402 250, height=40 pt]{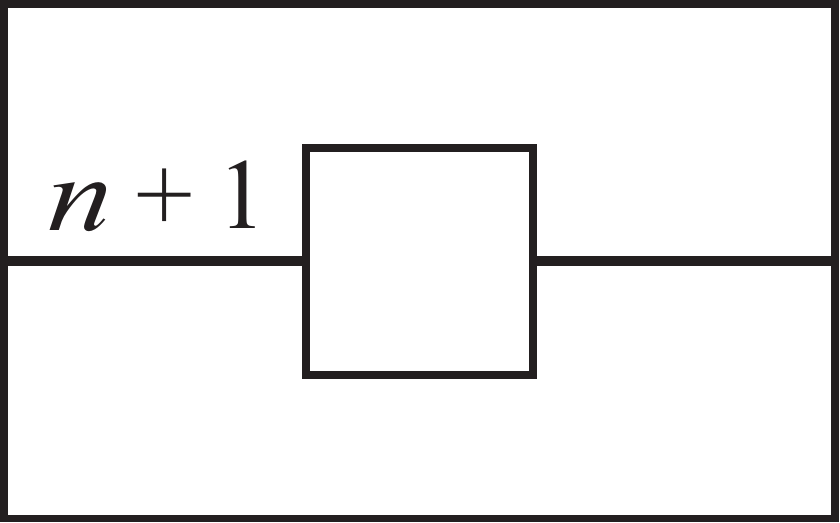}}
 =\raisebox{-17 pt}{\includegraphics[bb= 0 0 362 250, height=40 pt]{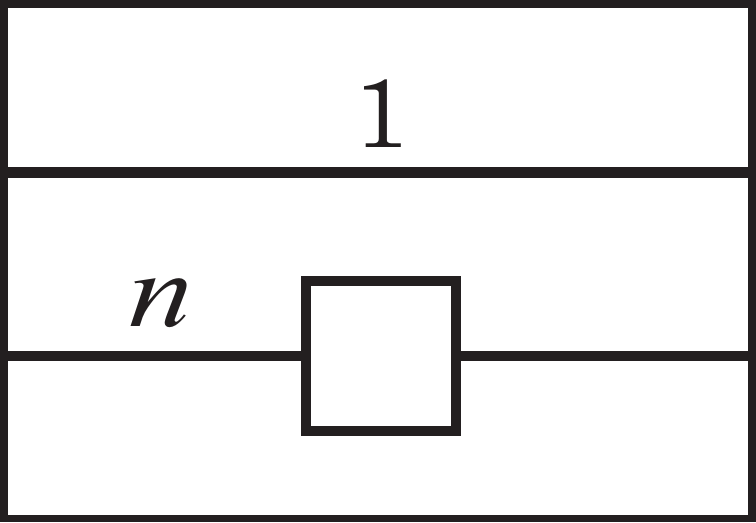}} 
  - \frac{\Delta_{n-1}}{\Delta_{n}}  \raisebox{-17 pt}{\includegraphics[bb= 0 0 415 250, height=40 pt]{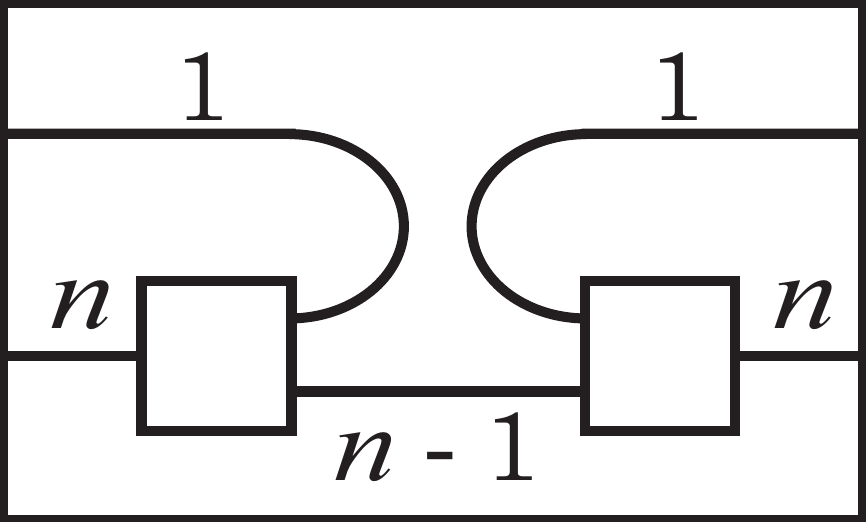}} \,\,,
$$
 where $JW_0$ is a blank diagram and
$$
 \Delta_{n}
 =\left< \raisebox{-19 pt}{\includegraphics[bb= 0 0 75 86, height=40 pt]{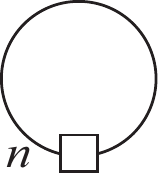}} \right>
 = (-1)^n\frac{A^{2(n+1)}-A^{-2(n+1)}}{A^2 - A^{-2}} = (-1)^n\frac{\sin\frac{\pi}{r}(n+1)}{\sin\frac{\pi}{r}}.$$
In these relations, the strands to which numbers are attached mean those numbers of parallel arcs. Note that $\Delta_{r-1}=0$.
\end{defn}
 \par
 The element $JW_n$ is called \textit{Jones-Wenzl idempotent} because of the following relation
\begin{eqnarray*}
 & &\raisebox{-17 pt}{\includegraphics[bb= 0 0 402 250, width=60 pt, height=40 pt]{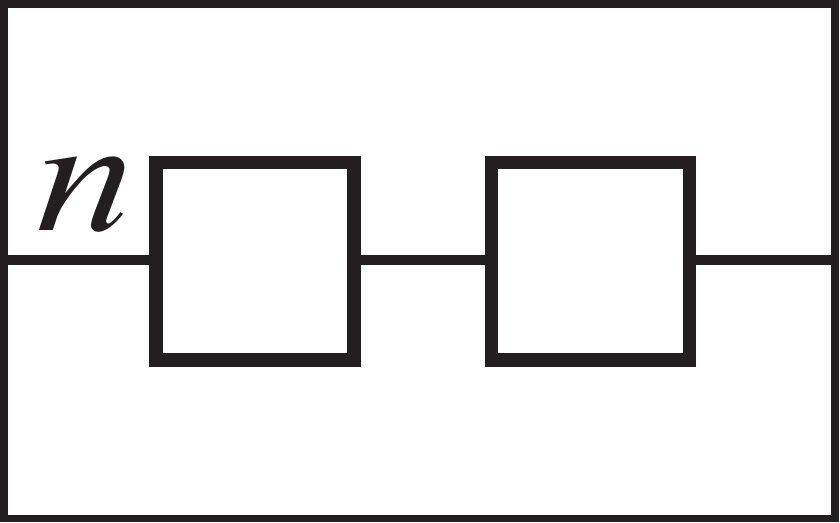}}
 =\raisebox{-17 pt}{\includegraphics[bb= 0 0 402 250, width=60 pt, height=40 pt]{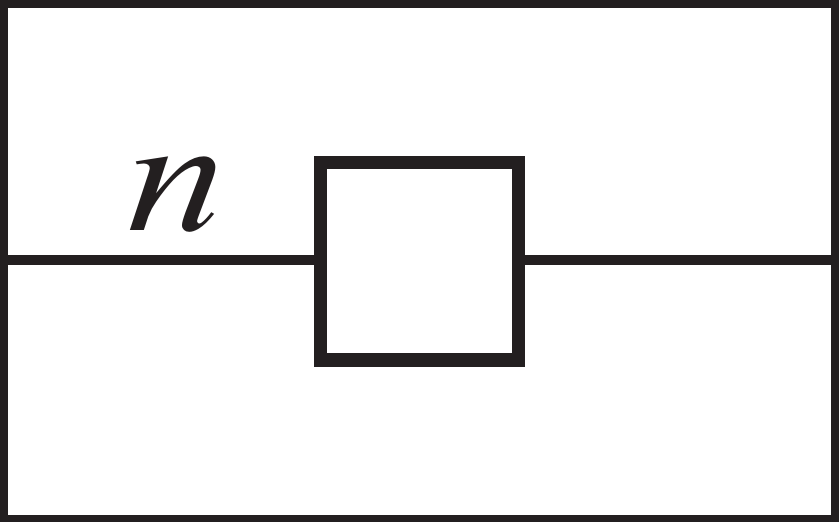}} \,\,.
 \end{eqnarray*}
     
\subsection{Definition of quantum $\mathcal{U}_q(\mathfrak{sl}_2)$ type invariants}
\par
Let $I$ be a set of integers $\{0, 1, \dots , r-2\}$. We call an element of $I$ a \textit{color}. A \textit{coloring} of a spatial graph $\Gamma$ is a map $E(\Gamma) \rightarrow I$, where $E(\Gamma)$ is a set of edges and circles of $\Gamma$. We represent a coloring by attaching colors to edges and circles. In \cite{YO}, Yokota defined invariants for colored spatial graphs by using the Kauffman bracket. We define quantum $\mathcal{U}_q(\mathfrak{sl}_2)$ type invariants for handlebody-links via Yokota's invariants.
 \par
From now on, we regard an $i$ colored edge or circle in a spatial trivalent graph diagram as $i$ parallel arcs in which the idempotent $JW_i$ is inserted. We also regard a trivalent vertex joined by three colored edges as a diagram of arcs (see Fig.~\ref{fig5}), where $x=(j+k-i)/2, y=(i+k-j)/2$ and $z=(i+j-k)/2$.
\begin{figure}[ht]
 $$  \raisebox{-19pt}{\includegraphics[bb= 0 0 80 69, height=50 pt]{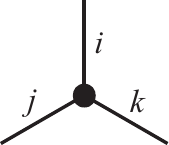}} \longrightarrow  \raisebox{-19pt}{\includegraphics[bb= 0 0 74 64, height=50 pt]{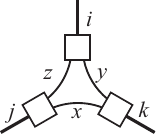}} $$
 \caption{A colored trivalent vertex and a corresponding diagram.} \label{fig5}
 \end{figure}
If $i$, $j$ and $k$ satisfy $i\leq j+k$, $j \leq i+k$, $k \leq i+j$ and $i+j+k \in 2\mathbb{Z}$, then we have the corresponding diagram in the right-hand side in Fig.~\ref{fig5}. Otherwise, we cannot realize the diagram. Moreover, if $i+j+k \leq 2(r-2)$, then the following value is non-zero.
$$\theta(i,j,k) = \left< \raisebox{-26 pt}{\includegraphics[bb= 0 0 109 87, width=60 pt]{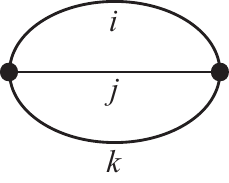}} \right> =  \left< \raisebox{-21 pt}{\includegraphics[bb= 0 0 88 71, width=60 pt]{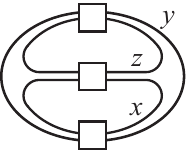}} \right> =  \frac{\Delta_{x+y+z}!\Delta_{x-1}!\Delta_{y-1}!\Delta_{z-1}!}{\Delta_{y+z-1}!\Delta_{x+z-1}!\Delta_{x+y-1}!} \,\,,$$
where $\Delta_n ! = \Delta_n \Delta_{n-1} \cdots \Delta_{0}$ and $\Delta_{-1}! = 1$. The value $\theta(i,j,k)$ is independent of the order of the triple $(i, j, k)$. We summarize these conditions.

\begin{defn}[Admissible condition]
  A triple $(i, j, k) \in I^3$ is called \textit{admissible} if $(i, j, k)$ satisfies the following conditions: 
 \begin{eqnarray*}
& &i\leq j+k,\,\,\,\,\, j \leq i+k,\,\,\,\,\, k \leq i+j,  \\
& &\hspace{-0.5cm} i+j+k \in 2\mathbb{Z}, \,\,\,\,\, i+j+k \leq 2(r-2).
 \end{eqnarray*}
A coloring for a spatial trivalent graph is called \textit{admissible} if the three colors at each vertex are admissible.
\end{defn}
 \par
Yokota's invariants for colored spatial graphs in $S^3$ are defined as follows: first they are defined for trivalent graphs and then they are generalized for arbitrary graphs.

\begin{defn}[Yokota's invariants \cite{YO}, see also \cite{BGM}]
 Let $\Gamma$ be a spatial trivalent graph in $S^3$. We fix a coloring to $\Gamma$ by attaching numbers $i_1, i_2, ... , i_n \in I$ to the edges and circles of $\Gamma$. If the coloring is admissible, for each vertex the triple $(i_a, i_b, i_c)$ of the edge colors is admissible. Let $D(i_1, i_2, ... , i_n)$ be a diagram of $\Gamma(i_1, i_2, ... , i_n)$ and $\overline{D}(i_1, i_2, ... , i_n)$ be the mirror image of $D(i_1, i_2, ... , i_n)$.  Then Yokota's invariants $\left<\,\cdot\,\right>_{Y}$ for the admissibly colored spatial trivalent graph $\Gamma(i_1, i_2, ... , i_n)$ are defined as follows: 
$$\left< \Gamma(i_1, i_2, ... , i_n) \right>_{Y} := \left< D(i_1, i_2, ... , i_n) \right> \left< \overline{D}(i_1, i_2, ... , i_n) \right> \mbox{\Large/} \hspace{-0.5cm} \prod_{\substack{\mbox{\footnotesize triples of colors}\\\mbox{\footnotesize at vertices}}} \hspace{-0.5cm} \theta(i_a, i_b, i_c).$$
If the coloring is not admissible, we define $\left< \Gamma(i_1, i_2, ... , i_n) \right>_{Y} :=0$.
 \par
Yokota's invariants are generalized to arbitrary graphs by the next relations at vertices. For an $n$-valent vertex ($n>3$),
\begin{equation} \label{eq1}
  \left< \, \raisebox{-19 pt}{\includegraphics[bb=0 0 500 200,width=70 pt]{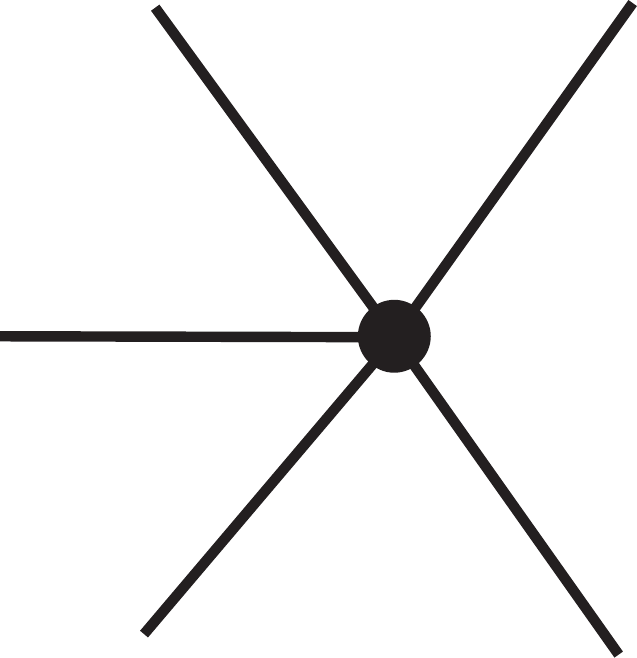}} \hspace{-0.8cm} \right>_{\mbox{$Y$}}
  =
  \sum_i \Delta_i 
  \left<\raisebox{-18 pt}{\includegraphics[bb= 0 0 133 83, width=70 pt]{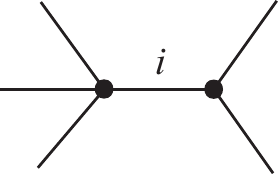}} \,\,\, \right>_{\mbox{$Y$}},
\end{equation}
 where the color $i$ runs over all admissible colors for the right-hand side diagram. For a 2-valent vertex, 
 \par
\[
  \left<  \raisebox{0 pt}{\includegraphics[bb= 0 0 118 29, width=70 pt]{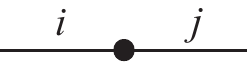}} \right>_{\mbox{\normalsize $Y$}}
  =
  \frac{\delta_{ij}}{\Delta_i}
  \left<\raisebox{2 pt}{\includegraphics[bb= 0 0 97 21, width=63 pt]{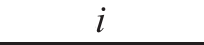}} \right>_{\mbox{$Y$}},
\]
and for a 1-valent vertex,
\[
  \left< \raisebox{-22.5 pt}{\includegraphics[bb= 0 0 108 91, width=60 pt]{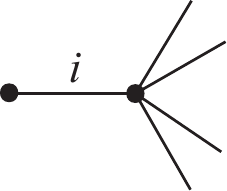}} \right>_{\mbox{$Y$}}
  =
  \delta_{i0} 
  \left<\raisebox{-22 pt}{\includegraphics[bb=0 0 500 200,width=70 pt]{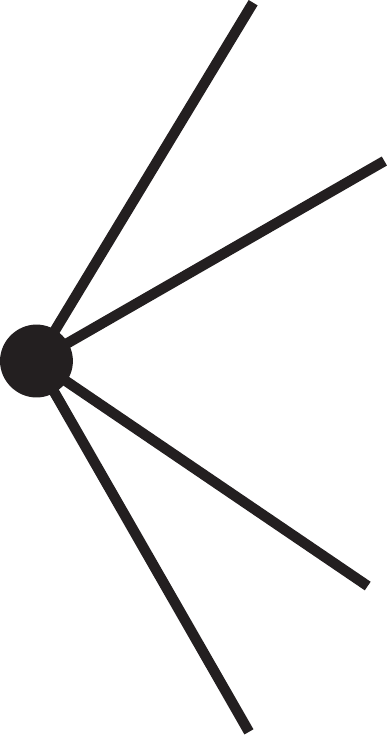}}  \hspace{-1.5cm} \right>_{\mbox{$Y$}}.
\]
\end{defn}

\begin{thm}[\cite{YO}]  \label{Thm1}
$\left< \,\cdot \,\right>_Y$ is invariant under the RI-RV moves of diagrams and the ways to extend an edge in Eq.~(\ref{eq1}). Therefore, $\left< \,\cdot \,\right>_Y$ is well-defined as an invariant of spatial graphs.
\end{thm}

\begin{rem}
 Originally, Yokota's invariants are defined with a general variable $A\in\mathbb{C}\setminus \{0\}$ and a set of colors $\mathbb{N} \cup \{0\}$. In this paper we fix $A={\rm e}^{2\pi i/4r}$ $(3\leq r \in \mathbb{N})$ and restrict the set of colors to $I$. Since $A^{-1}=\overline{A}$, $\displaystyle{\left< \overline{D} \right> = \overline{\left< D \right>}}$ and the definition of Yokota's invariants is rewritten as
$$\left< \Gamma(i_1, i_2, ... , i_n) \right>_{Y} = \left| \left< D(i_1, i_2, ... , i_n) \right> \right|^2 \hspace{0.2cm} \raisebox{-2pt}{\mbox{\Large/}} \hspace{-0.5cm} \prod_{\substack{\mbox{\footnotesize triples of colors}\\\mbox{\footnotesize at vertices}}} \hspace{-0.5cm} \theta(i_a, i_b, i_c).$$
\end{rem}
 \par
Now we define quantum $\mathcal{U}_q(\mathfrak{sl}_2)$ type invariants for handlebody-links in $S^3$ by taking linear sums of Yokota's invariants for all possible colorings with weights. 
\begin{defn}[Quantum $\mathcal{U}_q(\mathfrak{sl}_2)$ type invariants for handlebody-links]
 Let $J$ be a handlebody-link and $\Gamma$ be a spatial trivalent graph that represents $J$. We attach colors $i_k$ to edges and $j_l$ to circles. Then we define a value $\left< J\right>_H$ as follows: 
\begin{eqnarray*}
& &\hspace{-0.5cm}\left< J \right>_H:=\sum_{\substack{i_1, ... , i_n\\j_1, ... , j_m}} \Delta_{i_1} \cdots \Delta_{i_n} \left< \Gamma(i_1, ... ,i_n, j_1, ... ,j_m) \right>_Y\\
& &\hspace{0.5cm} = \sum_{\substack{i_1, ... , i_n\\j_1, ... , j_m}} \Delta_{i_1} \cdots \Delta_{i_n} \left| \left< D(i_1, ... , i_n, j_1, ... ,j_m) \right> \right|^2 \mbox{\Large/} \hspace{-0.6cm} \prod_{\substack{\mbox{\footnotesize triples of colors}\\\mbox{\footnotesize at vertices}}}  \hspace{-0.6cm} \theta(i_a, i_b, i_c) ,
\end{eqnarray*}
where $D(i_1, ... , i_n, j_1, ... ,j_m)$ is a diagram of $\Gamma(i_1, ... , i_n, j_1, ... ,j_m)$ and each $i_k$ or $j_l$ runs over all admissible colorings for $\Gamma$. Usually we put the diagram of the representing graph inside the bracket instead of the handlebody-link itself as $\left< D \right>_H := \left< J \right>_H$.
\end{defn}
\begin{thm} \label{Thm2}
 Let $D$ be a diagram of a spatial trivalent graph that represents a handlebody-link $J$. Then the value $\left< J \right>_H = \left< D \right>_H$ does not change under the Reidemeister moves RI-RVI. Hence $\left<\, \cdot \, \right>_H$ is an invariant for handlebody-links.
\end{thm}
\begin{proof} The invariance for RI-RV is derived from the invariance of Yokota's invariants. We show the invariance for RVI. By Eq.~(\ref{eq1}), 
$$
\sum_{i} \Delta_i \left< \, \raisebox{-17 pt}{\includegraphics[bb= 0 0 91 88, height=40 pt]{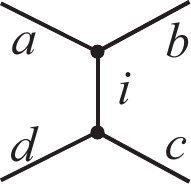}} \, \right>_{\mbox{$Y$}}
=
\left< \, \raisebox{-17 pt}{\includegraphics[bb= 0 0 105 92, height=40 pt]{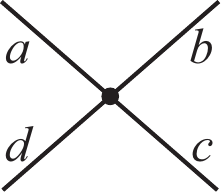}} \, \right>_{\mbox{$Y$}}
=
\sum_j \Delta_j \left< \, \raisebox{-17 pt}{\includegraphics[bb= 0 0 106 100, height=42 pt]{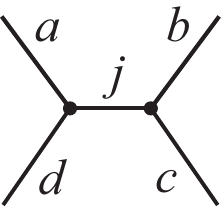}} \, \right>_{\mbox{$Y$}}.$$
Thus we have
\begin{eqnarray*}
& &\hspace{-0.6cm}\left< \, \raisebox{-17 pt}{\includegraphics[bb= 0 0 91 88, height=40 pt]{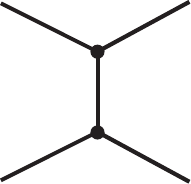}} \, \right>_{\mbox{$H$}}
=\sum_{a,b,c,d,i}  \Delta_a \Delta_b \Delta_c \Delta_d \Delta_i \left< \, \raisebox{-17 pt}{\includegraphics[bb= 0 0 91 88, height=40 pt]{IHmove02-2.pdf}} \, \right>_{\mbox{$Y$}}\\
& &\hspace{2.0cm}=\sum_{a,b,c,d,j}  \Delta_a \Delta_b \Delta_c \Delta_d \Delta_j \left< \, \raisebox{-17 pt}{\includegraphics[bb= 0 0 106 100, height=42 pt]{IHmove04.pdf}} \, \right>_{\mbox{$Y$}}
=\left< \, \raisebox{-17 pt}{\includegraphics[bb= 0 0 106 91, height=40 pt]{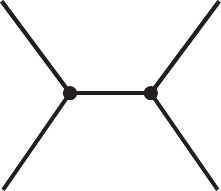}} \, \right>_{\mbox{$H$}}.
\end{eqnarray*}
\end{proof}

\section{Examples of the Invariants $\left< \, \cdot \, \right>_H$} \label{sec4}
 \par
In this section, we show some calculations of our invariants and a table of values of them for \textit{irreducible} genus 2 handlebody-knots up to six crossings classified in \cite{IKMS}. A handlebody-knot is called \textit{reducible} if it is represented by a spatial graph that has a cut edge. We also show some properties of our invariants.
 \par
First, we mention that $\left< \, \cdot \, \right>_H$ does not distinguish a handlebody-link from its mirror image. This is clear from the definition of $\left< \, \cdot \, \right>_H$.
\subsection{Relations}
 \par
We recall some formulas of the Kauffman bracket for diagrams including $JW_n$. For details, see \cite{KL, Lic}.
 \par
\noindent\textbf{Tetrahedron's edge.} (We rearrange the order of the variables of the formula in \cite{KL}.)
$$\left< \raisebox{-18pt}{\includegraphics[bb= 0 0 103 90, width=50 pt]{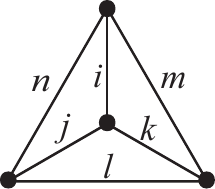}} \right> =: \left[ \begin{array}{ccc}  i & j & k \\ l & m & n \\ \end{array} \right] =  \frac{\mathcal{F}!}{\mathcal{E}!}\sum_{c \leq z \leq C}\frac{(-1)^z [z+1]!}{\prod_s [z-a_s]! \prod_t [b_t-z]!} ,$$
where 
\allowdisplaybreaks{\begin{eqnarray*}
 & & \hspace{-0.6cm} [n] = \frac{A^{2n} - A^{-2n}}{A^2 - A^{-2}} \left( = (-1)^{n-1} \Delta_{n-1} \right), \hspace{0.3cm} [n]! = [n][n-1] \cdots [1], \,\,\,\,\, [0]!=1,\\
 & & \hspace{1.2cm} \mathcal{F}! = \prod_{s, t} [b_t - a_s ]!, \hspace{0.7cm} \mathcal{E}! = [ i ]! [ j ]! [ k ]! [ l ]! [ m ]! [ n ]!,\\
 & & \hspace{1.3cm} a_1 = \frac{1}{2}(i + j + k), \,\,\,\,\,\,\,\,\, b_1 = \frac{1}{2}(i + j + l + m),\\
 & & \hspace{1.3cm} a_2 = \frac{1}{2}(i + m + n), \,\,\,\,\,\,\, b_2 = \frac{1}{2}(i + k + l + n),\\
 & & \hspace{1.3cm} a_3 = \frac{1}{2}(j + l + n), \,\,\,\,\,\,\,\,\,\, b_3 = \frac{1}{2}(j + k + m + n),\\
 & & \hspace{1.3cm} a_4 = \frac{1}{2}(k + l + m), \,\,\,\,\,\,\,\,\,\, c = \max\{a_s\}, \,\,\,\,\,\, C = \min\{b_t\}.\\
\end{eqnarray*}}
 \par
\noindent\textbf{Local change relations.}
 \begin{equation*} 
   \hspace{0.0cm}\left<\raisebox{-14 pt}{\includegraphics[bb= 0 0 114 53, width=60 pt]{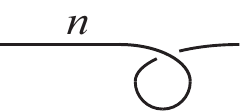}} \,\, \right>
  =
  (-1)^n A^{n^2+2n}
  \Biggr<\raisebox{2 pt}{\includegraphics[bb= 0 0 95 21, width=60 pt]{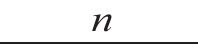}}  \Biggr>.
\end{equation*}
\vspace{0.5cm}
\begin{equation*} \label{eq3}
   \Biggr<\hspace{-0.1cm}\raisebox{-13 pt}{\includegraphics[bb= 0 0 103 58, width=60 pt]{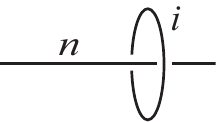}}\hspace{-0.1cm}\Biggr>
  =
  (-1)^i \frac{A^{2(n+1)(i+1)}-A^{-2(n+1)(i+1)}}{A^{2(n+1)}-A^{-2(n+1)}}
  \Biggr<\hspace{-0.1cm} \raisebox{3 pt}{\includegraphics[bb= 0 0 95 21, width=55 pt]{Equation13.pdf}}\hspace{-0.1cm} \Biggr>.
\end{equation*}

\vspace{0.5cm}
\begin{equation} \label{eq4}
   \left< \,\, \raisebox{-23 pt}{\includegraphics[bb= 0 0 101 110, height=50 pt]{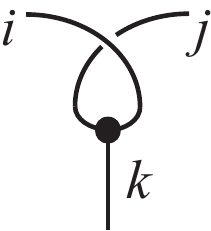}} \,\, \right>
  =
  \lambda^{ij}_k
  \left< \,\, \raisebox{-23 pt}{\includegraphics[bb= 0 0 62 70, height=50 pt]{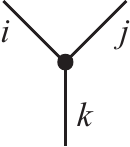}} \,\, \right>,
\end{equation}
where $\lambda^{ij}_k = (-1)^{(i + j - k)/2}A^{i + j - k + (i^2 + j^2 - k^2)/2}$.

\vspace{-0.5cm}
\begin{equation} \label{eq5}
   \left<\raisebox{-20 pt}{\includegraphics[bb= 0 0 119 77, width=70 pt]{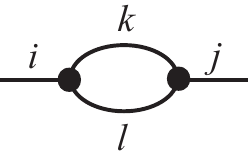}} \right>
  =
  \frac{\theta(i, k ,l)}{\Delta_i} \delta_{ij}
  \Biggr<\raisebox{3 pt}{\includegraphics[bb= 0 0 97 21, width=70 pt]{Equation12.pdf}} \Biggr>.
\end{equation}
\vspace{0.5cm}
\begin{equation}  \label{eq6}
   \left<\raisebox{-26 pt}{\includegraphics[bb= 0 0 148 102, width=80 pt]{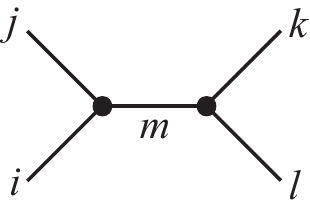}} \right>
  =
  \sum_n
  \left\{ \begin{array}{ccc}  i & j & m \\ k & l & n \\ \end{array} \right\}
  \left< \, \raisebox{-37 pt}{\includegraphics[bb= 0 0 90 136, width=53 pt]{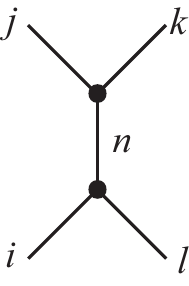}} \, \right>\,,
\end{equation}
where $\displaystyle{  \left\{ \begin{array}{ccc}  i & j & m \\ k & l & n \\ \end{array} \right\} = \frac{ \left[ \begin{array}{ccc}  i & j & m \\ k & l & n \\ \end{array} \right] \Delta_n}{\theta(i, l, n) \theta(j, k, n)}}.$
\vspace{0.5cm}
\begin{equation}  \label{eq7}
  \left< \,\, \raisebox{-21 pt}{\includegraphics[bb= 0 0 114 107, width=55 pt]{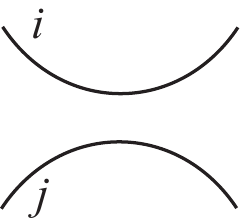}} \,\, \right>
  =
  \sum_k \frac{\Delta_k}{\theta(i, j, k)}
  \left<\raisebox{-26 pt}{\includegraphics[bb= 0 0 136 97, width=80 pt]{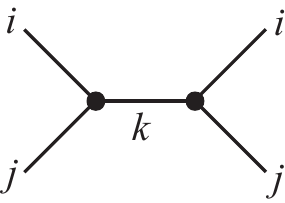}} \right>.
\end{equation}
In Eqs.~ (\ref{eq6}) and (\ref{eq7}), the colors $n$ and $k$ run over all admissible colors of the right-hand side diagrams.
 \par
\noindent\textbf{Omega element.}
Consider the skein space $\mathcal{S}(T)$ of the annulus $T$. Let $\alpha_n$ be an element of $\mathcal{S}(T)$ which is a collection of $n$ parallel closed curves along the band of $T$ in which $JW_n$ is inserted. An element $\Omega$ of $\mathcal{S}(T)$ is a linear sum of $\alpha_n$ with coefficients $\Delta_n$ for $0\leq n \leq r-2$. We will insert $\Omega$ in a component of diagrams along its framing. In diagrams, we represent $\Omega$ by a black box inserted in a component of the diagrams: 
\begin{equation} \label{eq8}
\Omega=\raisebox{-27 pt}{\includegraphics[bb= 0 0 266 266, width=60 pt]{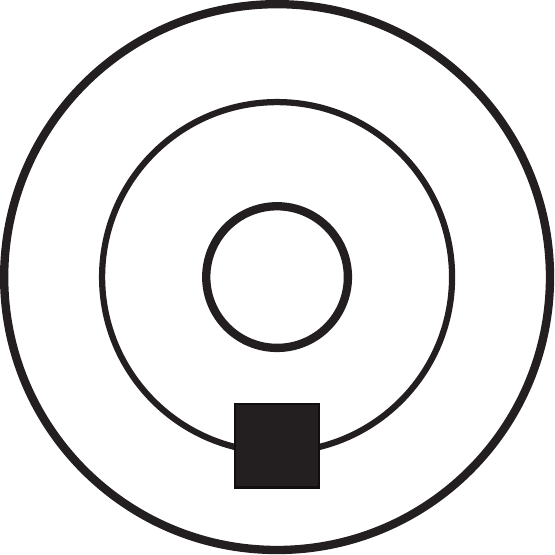}} \,= \sum_{n=0}^{r-2} \Delta_n \alpha_n =\sum_{n=0}^{r-2} \Delta_n \,\raisebox{-27 pt}{\begin{overpic}[bb= 0 0 266 266, width=60 pt]{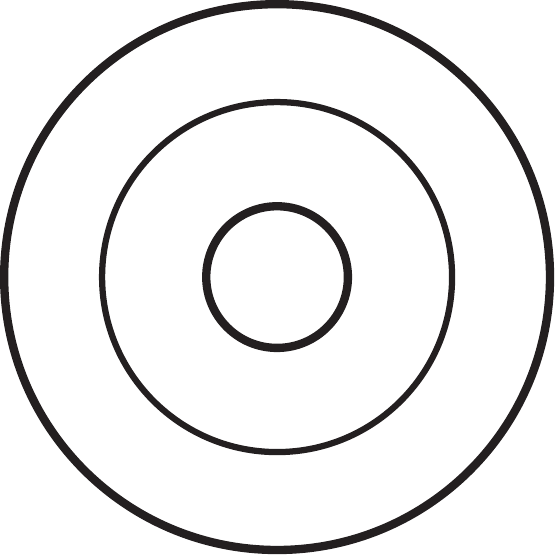}\put(44, 22){$n$}\end{overpic}}\,.
\end{equation}
Let $N$ be the value of the Kauffman bracket of the 0 framed trivial circle in which $\Omega$ is inserted: 
\begin{equation} \label{eq9}
N = \left<\raisebox{-22 pt}{\includegraphics[bb=  0 0 100 108, width=45 pt]{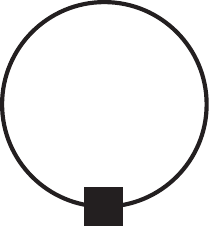}}\right> = \sum_{n=0}^{r-2} \Delta_n \left< \raisebox{-19 pt}{\begin{overpic}[bb= 0 0 171 171, width=45 pt]{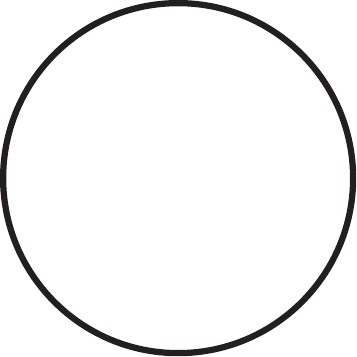}\put(43, 6){$n$} \end{overpic}} \right> =\sum_{n=0}^{r-2}\Delta_n^2= \frac{r}{2 \sin^2 \frac{\pi}{r}}.
\end{equation}
If a trivial $\Omega$ circle rounds strands in which $JW_n$ is inserted, the following relation holds:
\begin{equation} \label{eq10}
\left<\raisebox{-17 pt}{\begin{overpic}[bb= 0 0 292 189, width=60 pt]{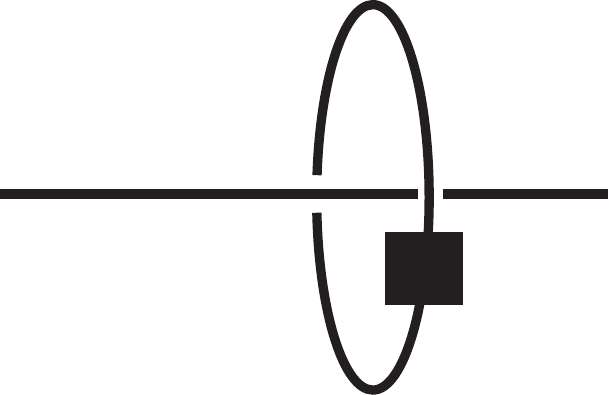}\put(25,37){$n$}\end{overpic}}\right> 
= \left\{
  \begin{array}{ll}
   N  \left<\raisebox{2 pt}{\includegraphics[bb= 0 0 95 21, width=60 pt]{Equation13.pdf}} \right> & \mbox{if } n = 0,\\
     & \\
    \hspace{1cm} 0 & \mbox{if } 1 \leq n \leq r-2.
  \end{array}
  \right.
\end{equation}
The following relation also holds: 
\begin{equation} \label{eq13}
 \left<\,\raisebox{-14 pt}{\begin{overpic}[bb= 0 0 131 63, width=70 pt]{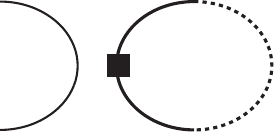}\put(20, 40){}\end{overpic}}\right> = \left< \,\raisebox{-19 pt}{\begin{overpic}[bb= 0 0 146 80, width=80 pt]{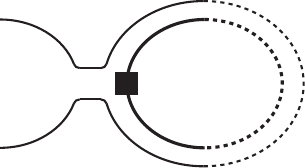}\put(20, 40){}\end{overpic}} \right> ,
\end{equation}
where the right-hand side diagram is a band-sum of a component and a parallel copy of another component in which $\Omega$ is inserted in the left-hand side diagram. 
 \subsection{Some examples}
\par 
First, we calculate our invariants for the genus 2 trivial handlebody-knot $0_1$ by two ways: via the theta curve diagram $D_1$ and via the handcuff diagram $D_2$. Secondly, we calculate our invariants for the genus 2 handlebody-knot called $4_1$ (Fig.~\ref{fig6}) in the table of handlebody-knots given in \cite{IKMS}.

\begin{figure}[ht]
$$\raisebox{-19 pt}{\includegraphics[bb= 0 0 74 75, width=60 pt]{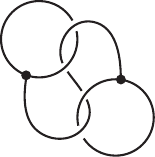}}$$
\caption{Handlebody-knot $4_1$.} \label{fig6}
\end{figure}

\begin{exam}[Theta curve]
\begin{eqnarray*}
& &\left< D_1 \right>_H
=\sum_{(i, j, k)} \Delta_{i} \Delta_{j} \Delta_{k} \left< \raisebox{-24 pt}{\includegraphics[bb= 0 0 109 87, height=45 pt]{thetacurve02.pdf}} \right>_{\mbox{$Y$}} \\
& & \hspace{1.1cm}= \sum_{(i, j, k)}\frac{\Delta_{i} \Delta_{j} \Delta_{k}}{ \theta(i, j, k)^2}\left< \raisebox{-24 pt}{\includegraphics[bb= 0 0 109 87, height=45 pt]{thetacurve02.pdf}} \right>\left< \raisebox{-24 pt}{\includegraphics[bb= 0 0 109 87, height=45 pt]{thetacurve02.pdf}} \right> \\
& & \hspace{1.1cm} = \sum_{(i, j, k)} \Delta_{i} \Delta_{j} \Delta_{k} 
=\sum_{i=0}^{r-2} \sum_{j=0}^{r-2} \Delta_{i}\Delta_{j} \sum_{\substack{k=|i-j|\\i+j+k \in 2\mathbb{Z}}}^{a}\Delta_{k}
=\frac{r^2}{4\sin^4 \frac{\pi}{r}},
\end{eqnarray*}\\
where in the summations $i$, $j$ and $k$ run over all colors such that  $(i,j,k)$ is admissible and $a=r-2-|(r-2)-(i+j)|$.
\end{exam}

\par
The value $\left<\, 0_1\right>_H$ is calculated easier via the handcuff diagram $D_2$ by using the following property.
\begin{prop}[Reducible splitting]  \label{Prop1}
 Let  $D$ be a diagram of a spatial trivalent graph. If $D$ has a cut edge (i.e. an edge that connects two disjoint subgraph diagrams that can be separated by a closed curve from each other), then we have the splitting relation
\[
\left< \raisebox{-17 pt}{\includegraphics[bb= 0 0 100 36, height=40 pt]{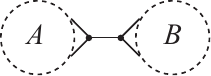}} \right>_{\mbox{$H$}}
=\left< \raisebox{-17 pt}{\includegraphics[bb= 0 0 65 53, height=40 pt]{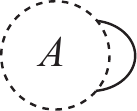}} \right>_{\mbox{$H$}}\left< \raisebox{-17 pt}{\includegraphics[bb= 0 0 65 53, height=40 pt]{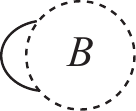}} \right>_{\mbox{$H$}}.
\]
\end{prop}
\begin{proof}
 The relation directly comes from the following relation of Yokota's invariants (see \cite{YO} Propositions 4.5 and 4.8): 
$$
\left< \raisebox{-17 pt}{\includegraphics[bb= 0 0 112 40, height=40 pt]{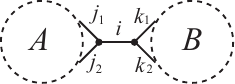}} \right>_{\mbox{$Y$}}
= \frac{\delta_{0i}}{\Delta_{j_1} \Delta_{k_1}} \left< \raisebox{-17 pt}{\includegraphics[bb= 0 0 59 45, height=40 pt]{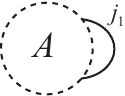}} \right>_{\mbox{$Y$}}\left< \raisebox{-17 pt}{\includegraphics[bb= 0 0 61 47, height=40 pt]{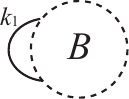}} \right>_{\mbox{$Y$}},
$$
where $(i, j_1, j_2)$ and $(i, k_1, k_2)$ are admissible. Therefore, if $i=0$, then $j_1=j_2$ and $k_1=k_2$.
\end{proof}

\begin{exam}[Handcuff]
\begin{eqnarray*}
& & \hspace{-0.8cm}\left< D_2 \right>_H
= \left< \raisebox{-17 pt}{\includegraphics[bb= 0 0 104 41, height=40 pt]{hand-cuffs01.pdf}} \right>_{\mbox{$H$}} 
= \left< \raisebox{-17 pt}{\includegraphics[bb= 0 0 73 73, height=40 pt]{skein-relation04.pdf}} \right>_{\mbox{$H$}}\left< \raisebox{-17 pt}{\includegraphics[bb= 0 0 73 73, height=40 pt]{skein-relation04.pdf}} \right>_{\mbox{$H$}} \\
\\
& & \hspace{0.3cm}=\sum_{i=0}^{r-2} \left< \raisebox{-17 pt}{\includegraphics[bb= 0 0 70 70, height=40 pt]{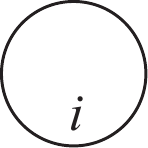}} \right>^{\mbox{2}}\,\,\sum_{j=0}^{r-2} \left< \raisebox{-17 pt}{\includegraphics[bb= 0 0 77 77, height=40 pt]{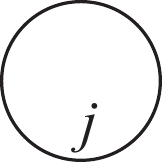}} \right>^{\mbox{2}}
 = \sum_{i=0}^{r-2} \Delta_{i}^2 \sum_{j=0}^{r-2} \Delta_{j}^2
 =\frac{r^2}{4\sin^4 \frac{\pi}{r}}.
\end{eqnarray*}
\end{exam} 
 \par
\begin{exam}[Handlebody-knot $4_1$]
\par
$$\left< 4_1 \right>_H
=\sum_{\substack{(i, i, k)\\(j,j,k)}} \Delta_{i} \Delta_{j} \Delta_{k} \left< \raisebox{-16 pt}{\includegraphics[bb= 0 0 80 75, width=40 pt]{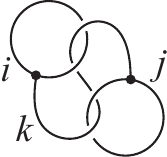}} \right>_{\mbox{$Y$}}
=\sum_{\substack{(i, i, k)\\(j,j,k)}} \frac{\Delta_{i} \Delta_{j} \Delta_{k}}{\theta(i, i, k) \theta(j, j, k)}\left|\left< \raisebox{-16 pt}{\includegraphics[bb= 0 0 80 75,width=40 pt]{hand-cuffs-1-left.pdf}} \right>\right|^{\mbox{2}}.$$
Here
\allowdisplaybreaks{\begin{eqnarray*}
& &\hspace{0.0cm}\left< \raisebox{-16 pt}{\includegraphics[bb= 0 0 80 75, width=40 pt]{hand-cuffs-1-left.pdf}} \right>
\overset{(\ref{eq7})}{=}\sum_{\substack{(i,k,l)\\(j,k,m)}}\frac{\Delta_{l} \Delta_{m}}{\theta(i, k, l) \theta(j, k, m)}\left< \raisebox{-15 pt}{\includegraphics[bb= 0 0 84 75, width=40 pt]{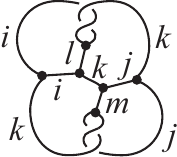}} \right>\\
& &\overset{(\ref{eq4})}{=}\sum_{\substack{(i,k,l)\\(j,k,m)}}\frac{\Delta_{l} \Delta_{m} \left(\lambda_{l}^{ik} \right)^{-2} \left(\lambda_{m}^{jk} \right)^{-2}}{\theta(i, k, l) \theta(j, k, m)}\left< \raisebox{-15 pt}{\includegraphics[bb= 0 0 84 75, width=40 pt]{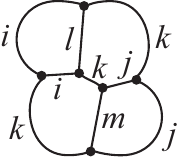}} \right>\\
& &\overset{(\ref{eq6})}{=}\sum_{\substack{(i,k,l),\,(j,k,m)\\(i,i,s),\,(k,k,s)\\(j,j,t),\,(k,k,t)}}\frac{\Delta_{l} \Delta_{m} \Delta_{s} \Delta_{t} \left(\lambda_{l}^{ik} \right)^{-2} \left(\lambda_{m}^{jk} \right)^{-2}}{\theta(i, k, l) \theta(j, k, m) \theta(i, i, s) \theta(k, k, s) \theta(j, j, t) \theta(k, k, t)}\\
& &\hspace{6.0cm}\times\left[ \begin{array}{ccc}  i & k & l \\ k & i & s \\ \end{array} \right]\left[ \begin{array}{ccc}  j & k & m \\ k & j & t \\ \end{array} \right]\left< \raisebox{-15 pt}{\includegraphics[bb= 0 0 86 77, width=40 pt]{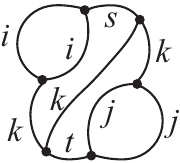}} \right>\\
& &\overset{(\ref{eq5})}{=}\sum_{\substack{(i,k,l),\,(j,k,m)\\(i,i,s),\,(k,k,s)\\(j,j,t),\,(k,k,t)}}\frac{\Delta_{l} \Delta_{m} \Delta_{s} \Delta_{t} \left(\lambda_{l}^{ik} \right)^{-2} \left(\lambda_{m}^{jk} \right)^{-2}}{\theta(i, k, l) \theta(j, k, m) \theta(i, i, s) \theta(k, k, s) \theta(j, j, t) \theta(k, k, t)}\\
& &\hspace{3.0cm}\times\left[ \begin{array}{ccc}  i & k & l \\ k & i & s \\ \end{array} \right]\left[ \begin{array}{ccc}  j & k & t \\ k & j & m \\ \end{array} \right]\frac{\delta_{ks} \delta_{kt} \theta(i, i, k) \theta(j, j, k)}{\Delta_k^2}\theta(k,k,k)\\
& &\hspace{0.2cm}= \sum_{\substack{(i,k,l),\,(j,k,m)\\(i,i,k),\,(j,j,k)\\(k,k,k)}} \frac{\Delta_{l}\Delta_{m}  \left(\lambda_{l}^{ik} \right)^{-2} \left(\lambda_{m}^{jk} \right)^{-2}}{\theta(i, k, l) \theta(j, k, m) \theta(k, k, k)}\left[ \begin{array}{ccc}  i & k & l \\ k & i & k \\ \end{array} \right]\left[ \begin{array}{ccc}  j & k & m \\ k & j & k \\ \end{array} \right].
\end{eqnarray*}}
Thus we have
\begin{eqnarray*}
& &\hspace{0.0cm}\left< 4_1 \right>_H
= \sum_{\substack{(i, i, k)\\(j,j,k)}} \frac{\Delta_{i} \Delta_{j} \Delta_{k}}{\theta(i, i, k) \theta(j, j, k)}\\
& &\hspace{2.5cm}\times \Biggr| \sum_{\substack{(i,k,l),\,(j,k,m)\\(i,i,k),\,(j,j,k)\\(k,k,k)}} \frac{\Delta_{l}\Delta_{m}  \left(\lambda_{l}^{ik} \right)^{-2} \left(\lambda_{m}^{jk} \right)^{-2}}{\theta(i, k, l) \theta(j, k, m) \theta(k, k, k)}\left[ \begin{array}{ccc}  i & k & l \\ k & i & k \\ \end{array} \right]\left[ \begin{array}{ccc}  j & k & m \\ k & j & k \\ \end{array} \right] \Biggr|^{\mbox{2}}.
\end{eqnarray*}
\end{exam}

\subsection{Table of values of $\left< \, \cdot \, \right>_H$}
 \par
We did numerical calculations of the values of $\left< \, \cdot \, \right>_H$ for genus 2 irreducible handlebody-knots up to six crossings classified in \cite{IKMS}. The results for $3 \leq r \leq 10$ are in Table \ref{table1}. We show them to the sixth decimal places.
We comment on some features of Table \ref{table1}. When $r=3$ and $4$, all values of $\left< \, \cdot \, \right>_H$ are the same, 4 and 16 respectively. In case $r=3$, we have the following proposition.

\begin{prop}
Let $J$ be a handlebody-link and $D$ be a diagram of $J$. By a crossing change at a crossing point of $D$, we have another diagram $D'$ for a handlebody-link $J'$. If $r=3$, $\left< J \right>_H = \left< J' \right>_H$ i.e. $\left< \, \cdot \, \right>_H$ can not detect a crossing change. Hence for any $J$, $\left< J \right>_H = \left< J_0 \right>_H = 2^g$, where $g$ is the sum of the genera of the components of $J$ and $J_0$ is the trivial handlebody-link that has the same genera of the components with $J$. The second equation comes from Proposition \ref{Prop1} and $\left< \bigcirc \right>_H =2$ when $r=3$.
\end{prop}
\begin{proof}
Since $r=3$, we have $I=\{0,1\}$, $A={\rm e}^{\pi i/6}$ and $\Delta_1=-1$. We check that the difference of the values of the Kauffman brackets in the definition of $\left< \, \cdot \, \right>_H$ for the two diagrams which differ only at a crossing point $c$ is 0. We focus on the neighborhood of $c$. If one color of the two strands of $c$ is 0, the crossing change causes nothing. Therefore, it is enough to check the case where the both colors are 1.
\begin{eqnarray*}
& &\hspace{-0.1cm} \Biggl< \hspace{0.1cm} \raisebox{-10 pt}{\includegraphics[bb= 0 0 100 100, width=25 pt]{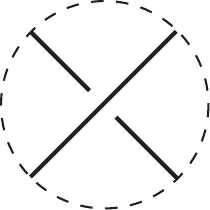}} \hspace{0.1cm} \Biggr>\Biggl< \hspace{0.1cm} \raisebox{-10 pt}{\includegraphics[bb= 0 0 100 100, width=25 pt]{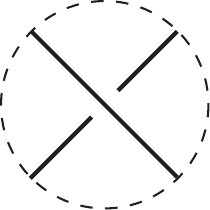}} \hspace{0.1cm} \Biggr> = A \Biggl< \hspace{0.1cm} \raisebox{-10 pt}{\includegraphics[bb= 0 0 100 100, width=25 pt]{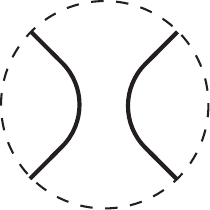}} \hspace{0.1cm} \Biggr>\Biggl< \hspace{0.1cm} \raisebox{-10 pt}{\includegraphics[bb= 0 0 100 100, width=25 pt]{skein-relation01-s.pdf}} \hspace{0.1cm} \Biggr>+ A^{-1}\Biggl< \hspace{0.1cm} \raisebox{-10 pt}{\includegraphics[bb= 0 0 100 100, width=25 pt]{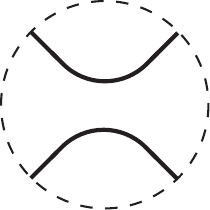}} \hspace{0.1cm} \Biggr>\Biggl< \hspace{0.1cm} \raisebox{-10 pt}{\includegraphics[bb= 0 0 100 100, width=25 pt]{skein-relation01-s.pdf}} \hspace{0.1cm} \Biggr> \\
& & \hspace{3.3cm} = A^2 \Biggl< \hspace{0.1cm} \raisebox{-10 pt}{\includegraphics[bb= 0 0 100 100, width=25 pt]{skein-relation02-s.pdf}} \hspace{0.1cm} \Biggr>\Biggl< \hspace{0.1cm} \raisebox{-10 pt}{\includegraphics[bb= 0 0 100 100, width=25 pt]{skein-relation03-s.pdf}} \hspace{0.1cm} \Biggr>+\Biggl< \hspace{0.1cm} \raisebox{-10 pt}{\includegraphics[bb= 0 0 100 100, width=25 pt]{skein-relation02-s.pdf}} \hspace{0.1cm} \Biggr>\Biggl< \hspace{0.1cm} \raisebox{-10 pt}{\includegraphics[bb= 0 0 100 100, width=25 pt]{skein-relation02-s.pdf}} \hspace{0.1cm} \Biggr>\\
& & \hspace{4.4cm} +\Biggl< \hspace{0.1cm} \raisebox{-10 pt}{\includegraphics[bb= 0 0 100 100, width=25 pt]{skein-relation03-s.pdf}} \hspace{0.1cm} \Biggr>\Biggl< \hspace{0.1cm} \raisebox{-10 pt}{\includegraphics[bb= 0 0 100 100, width=25 pt]{skein-relation03-s.pdf}} \hspace{0.1cm} \Biggr>+ A^{-2}\Biggl< \hspace{0.1cm} \raisebox{-10 pt}{\includegraphics[bb= 0 0 100 100, width=25 pt]{skein-relation03-s.pdf}} \hspace{0.1cm} \Biggr>\Biggl< \hspace{0.1cm} \raisebox{-10 pt}{\includegraphics[bb= 0 0 100 100, width=25 pt]{skein-relation02-s.pdf}} \hspace{0.1cm} \Biggr>.
\end{eqnarray*}
Similarly, we have
\begin{eqnarray*}
& & \hspace{-0.1cm}\Biggl< \hspace{0.1cm} \raisebox{-10 pt}{\includegraphics[bb= 0 0 100 100, width=25 pt]{skein-relation01-s.pdf}} \hspace{0.1cm} \Biggr>\Biggl< \hspace{0.1cm} \raisebox{-10 pt}{\includegraphics[bb= 0 0 100 100, width=25 pt]{skein-relation00-s.pdf}} \hspace{0.1cm} \Biggr> = A^2 \Biggl< \hspace{0.1cm} \raisebox{-10 pt}{\includegraphics[bb= 0 0 100 100, width=25 pt]{skein-relation03-s.pdf}} \hspace{0.1cm} \Biggr>\Biggl< \hspace{0.1cm} \raisebox{-10 pt}{\includegraphics[bb= 0 0 100 100, width=25 pt]{skein-relation02-s.pdf}} \hspace{0.1cm} \Biggr> + \Biggl< \hspace{0.1cm} \raisebox{-10 pt}{\includegraphics[bb= 0 0 100 100, width=25 pt]{skein-relation03-s.pdf}} \hspace{0.1cm} \Biggr>\Biggl< \hspace{0.1cm} \raisebox{-10 pt}{\includegraphics[bb= 0 0 100 100, width=25 pt]{skein-relation03-s.pdf}} \hspace{0.1cm} \Biggr> \\
& & \hspace{4.3cm} + \Biggl< \hspace{0.1cm} \raisebox{-10 pt}{\includegraphics[bb= 0 0 100 100, width=25 pt]{skein-relation02-s.pdf}} \hspace{0.1cm} \Biggr>\Biggl< \hspace{0.1cm} \raisebox{-10 pt}{\includegraphics[bb= 0 0 100 100, width=25 pt]{skein-relation02-s.pdf}} \hspace{0.1cm} \Biggr>+ A^{-2} \Biggl< \hspace{0.1cm} \raisebox{-10 pt}{\includegraphics[bb= 0 0 100 100, width=25 pt]{skein-relation02-s.pdf}} \hspace{0.1cm} \Biggr>\Biggl< \hspace{0.1cm} \raisebox{-10 pt}{\includegraphics[bb= 0 0 100 100, width=25 pt]{skein-relation03-s.pdf}} \hspace{0.1cm} \Biggr>.
\end{eqnarray*}
To show the difference between the above two equations is 0, it is sufficient to check that
\begin{equation} \label{eq14}
\Biggl< \hspace{0.1cm} \raisebox{-10 pt}{\includegraphics[bb= 0 0 100 100, width=25 pt]{skein-relation02-s.pdf}} \hspace{0.1cm} \Biggr>\Biggl< \hspace{0.1cm} \raisebox{-10 pt}{\includegraphics[bb= 0 0 100 100, width=25 pt]{skein-relation03-s.pdf}} \hspace{0.1cm} \Biggr>
-\Biggl< \hspace{0.1cm} \raisebox{-10 pt}{\includegraphics[bb= 0 0 100 100, width=25 pt]{skein-relation03-s.pdf}} \hspace{0.1cm} \Biggr>\Biggl< \hspace{0.1cm} \raisebox{-10 pt}{\includegraphics[bb= 0 0 100 100, width=25 pt]{skein-relation02-s.pdf}} \hspace{0.1cm} \Biggr>=0.
\end{equation}
Using the second relation of the skein space, we smooth all crossings in the outside of the neighborhood of $c$. There are four possible cases for the outside arcs which connect the endpoints of the arcs in the neighborhood of $c$.
$$\Biggl< \hspace{0.15cm} \raisebox{-19 pt}{\includegraphics[bb= 0 0 100 198, width=22 pt]{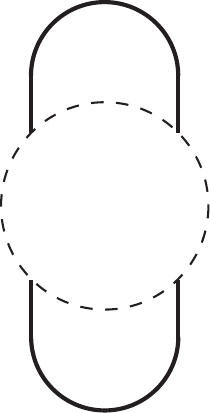}} \hspace{0.15cm} \Biggr>\Biggl< \hspace{0.15cm} \raisebox{-19 pt}{\includegraphics[bb= 0 0 100 198, width=22 pt]{skein-relation05-s2.pdf}} \hspace{0.15cm} \Biggr>, 
\hspace{0.5cm}\Biggl< \hspace{0.15cm} \raisebox{-19 pt}{\includegraphics[bb= 0 0 100 198, width=22 pt]{skein-relation05-s2.pdf}} \hspace{0.15cm} \Biggr>\Biggl< \hspace{-0.1cm} \raisebox{-8 pt}{\includegraphics[bb= 0 0 198 100, height=22 pt]{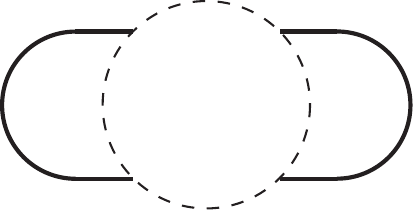}} \hspace{-0.1cm} \Biggr>,$$
$$ \Biggl< \hspace{-0.1cm} \raisebox{-8 pt}{\includegraphics[bb= 0 0 198 100, height=22 pt]{skein-relation06-s2.pdf}} \hspace{-0.1cm} \Biggr>\Biggl< \hspace{0.15cm} \raisebox{-19 pt}{\includegraphics[bb= 0 0 100 198, width=22 pt]{skein-relation05-s2.pdf}} \hspace{0.15cm} \Biggr>,
\hspace{0.5cm} \Biggl< \hspace{-0.1cm} \raisebox{-8 pt}{\includegraphics[bb= 0 0 198 100, height=22 pt]{skein-relation06-s2.pdf}} \hspace{-0.1cm} \Biggr>\Biggl< \hspace{-0.1cm} \raisebox{-8 pt}{\includegraphics[bb= 0 0 198 100, height=22 pt]{skein-relation06-s2.pdf}} \hspace{-0.1cm} \Biggr>.$$
For any case, the difference of the numbers of the circles in Eq.~(\ref{eq14}) is even. Since $\Delta_1=-1$, Eq.~(\ref{eq14}) holds. 
\end{proof}

\begin{table}[bht] \label{table1}
{\small
{\begin{tabular}{|c||c|c|c|c|c|c|}
\hline
r & $0_1$ & $4_1$ & $5_1$ & $5_2$ & $5_3$ & $5_4$ \\
\hline 
3 & 4 & 4 & 4 & 4 & 4 & 4 \\
\hline
4 & 16 & 16 & 16 & 16 & 16 & 16  \\
\hline
5 & \hphantom{000}52.360679 & \hphantom{000}84.721359 & \hphantom{000}32.360679 & \hphantom{000}84.721359 & \hphantom{000}64.721359 & \hphantom{000}84.721359 \\
\hline
6 & 144  & 216 & 144 & 216 & 144 & 144 \\
\hline
7 & \hphantom{00}345.654799 & \hphantom{00}499.485657 & \hphantom{00}376.122887 & \hphantom{00}499.485657 & \hphantom{00}351.689378 & \hphantom{00}567.946830 \\
\hline
8 & \hphantom{00}746.038672 & \hphantom{0}1364.077344 & \hphantom{00}927.058008 & \hphantom{0}1364.077344 & \hphantom{00}799.058008 & \hphantom{0}1982.116016\\
\hline
9 & \hphantom{0}1479.852018 & \hphantom{0}4090.669434 & \hphantom{0}1387.831255 & \hphantom{0}2578.238436 & \hphantom{0}1719.664275 & \hphantom{0}4393.545672\\
\hline
10 & \hphantom{0}2741.640786 & \hphantom{0}8429.906831 & \hphantom{0}2618.033989 & \hphantom{0}4636.067977 & \hphantom{0}3218.033989 & 10319.349550\\
\hline
\cline{1-7}
r &  $6_1$ & $6_2$ & $6_3$ & $6_4$ & $6_5$ & $6_6$ \\
\hline 
3 & 4 & 4 & 4 & 4 & 4 & 4\\
\hline
4 & 16 & 16 & 16 & 16 & 16 & 16\\
\hline
5 & \hphantom{000}64.721359 & \hphantom{000}64.721359 & \hphantom{000}44.721359 & \hphantom{000}32.360679 & \hphantom{000}64.721359 & \hphantom{000}64.721359\\
\hline
6 & 144 & 144 & 144 & 144 & 144 & 144\\
\hline
7 & \hphantom{00}406.590975 & \hphantom{00}351.689378 & \hphantom{00}302.822359 & \hphantom{00}376.122887 & \hphantom{00}351.689378 & \hphantom{00}351.689378\\
\hline
8 & \hphantom{0}1108.077344 & \hphantom{00}799.058008 & \hphantom{00}543.058008 & \hphantom{00}927.058008 & \hphantom{00}799.058008 & \hphantom{00}799.058008\\
\hline
9 & \hphantom{0}2323.704917 & \hphantom{0}1719.664275 & \hphantom{0}1127.311315 & \hphantom{0}1387.831255 & \hphantom{0}1719.664275 & \hphantom{0}1719.664275 \\
\hline
10 & \hphantom{0}4712.461179 & \hphantom{0}3218.033989 & 2200 & \hphantom{0}2618.033989 & \hphantom{0}3218.033989 & \hphantom{0}3218.033989\\
\hline
\cline{1-7}
r & $6_7$ & $6_8$ & $6_9$ & $6_{10}$ & $6_{11}$ & $6_{12}$\\
\hline
3  & 4 & 4 & 4 & 4 & 4 & 4 \\
\hline
4  & 16 & 16 & 16 & 16 & 16 & 16\\
\hline
5  & \hphantom{000}57.082039 & \hphantom{000}44.721359 & \hphantom{000}77.082039 & \hphantom{000}97.082039 & \hphantom{000}44.721359 & \hphantom{000}64.721359\\
\hline
6  & 144 & 144 & 216 & 144 & 144 & 144\\
\hline
7  & \hphantom{00}638.798446 & \hphantom{00}401.751619 & \hphantom{00}544.708542  & \hphantom{00}420.150550 & \hphantom{00}272.354271 & \hphantom{00}628.883006\\
\hline
8  & \hphantom{0}1214.116016 & \hphantom{00}980.077343 & \hphantom{0}1470.116016 & \hphantom{0}1470.116016 & \hphantom{00}415.058008 & \hphantom{0}1108.077344\\
\hline
9 & \hphantom{0}2011.724100 & \hphantom{0}1774.998117 & \hphantom{0}2568.781545 & \hphantom{0}2753.108230 & \hphantom{0}1004.001433 & \hphantom{0}2131.161018\\
\hline
10 & \hphantom{0}3905.572809 & \hphantom{0}2923.606798 & \hphantom{0}4970.820393 & \hphantom{0}5912.461179 & 2200 & \hphantom{0}4388.854382\\
\hline
\cline{1-5}
r  & $6_{13}$ & $6_{14}$ & $6_{15}$ & $6_{16}$ \\
\cline{1-5}
3 & 4 & 4 & 4 & 4 \\
\cline{1-5}
4  & 16 & 16 & 16 & 16\\
\cline{1-5}
5  & \hphantom{000}84.721359 & \hphantom{000}84.721359 & \hphantom{000}84.721359 & \hphantom{000}64.721359 \\
\cline{1-5}
6  & 216 & 288 & 288 & 144 \\
\cline{1-5}
7  & \hphantom{00}499.485657 & \hphantom{00}721.777687 & \hphantom{00}721.777687 & \hphantom{00}872.259607 \\
\cline{1-5}
8  & \hphantom{0}1364.077344 & \hphantom{0}1982.116016 & \hphantom{0}1982.116016 & \hphantom{0}2706.193359 \\
\cline{1-5}
9  & \hphantom{0}2578.238436 & \hphantom{0}3879.598459 & \hphantom{0}3879.598459 & \hphantom{0}3367.578579 \\
\cline{1-5}
10 & \hphantom{0}4636.067977 & \hphantom{0}7577.708764 & \hphantom{0}7577.708764 & \hphantom{0}5018.033989 \\
\cline{1-5}
\end{tabular}}
}
\caption{The values of $\left< \, \cdot \, \right>_H$.} \label{table1}
\end{table}

\par
Comparing $4_1$ and $5_2$, the values are the same up to $r=7$ and differ after that. We do not know when values of two handlebody-links begin to differ. In the table, there are four sets we could not distinguish by our invariants for $r \leq 10$. These sets are $\{ 5_1, 5_2, 6_4, 6_{13} \}$, $\{ 5_3, 6_2 \}$, $\{ 6_5, 6_6 \}$ and $\{ 6_{14}, 6_{15} \}$. There are two relations of handlebody-links that $\left< \, \cdot \, \right>_H$ cannot detect.

\begin{prop} \label{Prop2}
 Let $J_1$ and $J_2$ be two handlebody-links and $D_1$ and $D_2$ be their diagrams. If $D_1$ is divided into two parts by a simple circle that intersects $D_1$ at at most three points and $D_1$ changes to $D_2$ by taking the mirror image of one of the parts, then $\left< J_1 \right>_H = \left< J_2 \right>_H$.
\end{prop}
\begin{proof}
From Eqs.~(\ref{eq7}) and (\ref{eq10}), we have the following two relations of the Kauffman bracket for parts of diagrams: 
\begin{equation}\label{eq11} \left<\raisebox{-22 pt}{\begin{overpic}[bb= 0 0 261 197, width=60 pt]{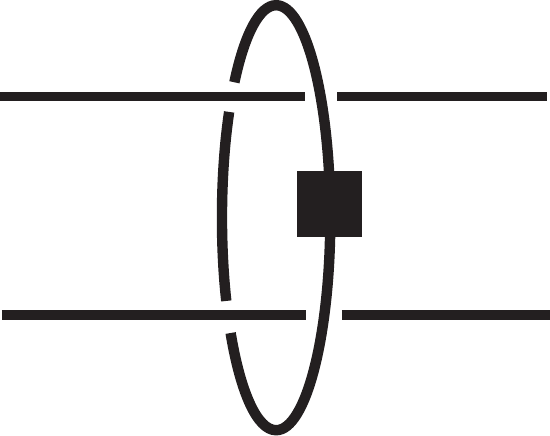}\put(18, 30){$j$}\put(19, 67){$i$}\end{overpic}}\right> =\frac{\delta_{ij}}{\Delta_i} \left<\raisebox{-17 pt}{\begin{overpic}[bb= 0 0 266 191, width=55 pt]{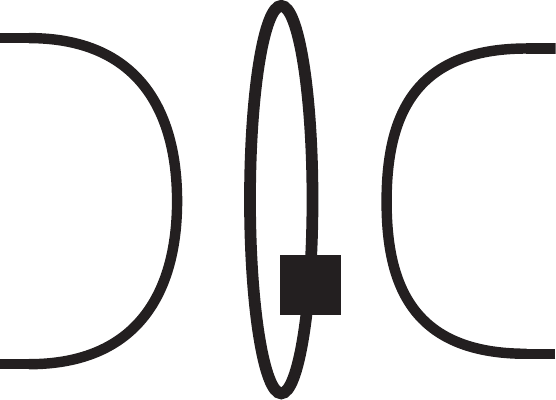}\put(15, 17){$i$}\put(80, 17){$i$}\end{overpic}}\right>.\end{equation}
\begin{equation}\label{eq12}\left<\raisebox{-20 pt}{\begin{overpic}[bb= 0 0 263 209, width=60 pt]{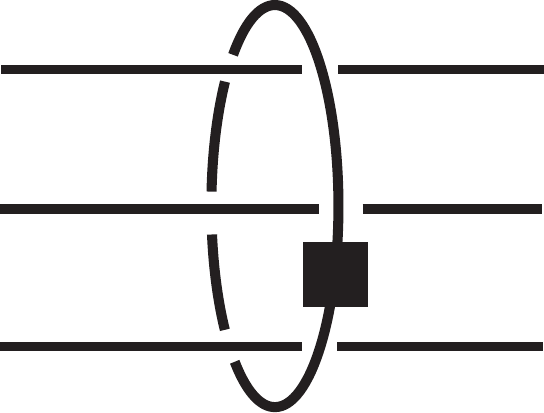}\put(19, 70){$i$}\put(18, 44){$j$}\put(17, 14){$k$}\end{overpic}}\right>=\frac{1}{\theta(i, j, k)} \left<\raisebox{-17 pt}{\begin{overpic}[bb= 0 0 261 191, width=55 pt]{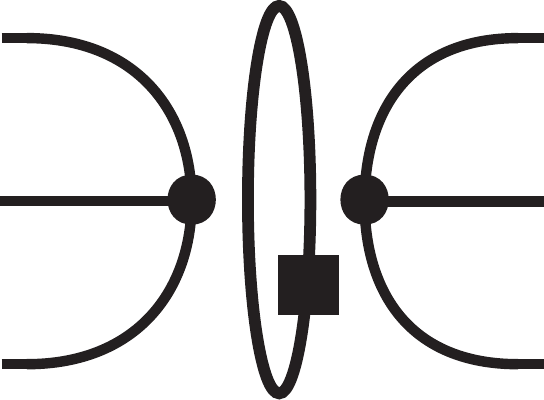}\put(7, 71){$i$}\put(6, 44){$j$}\put(5, 11){$k$}\put(88, 71){$i$}\put(87, 44){$j$}\put(86, 11){$k$}\end{overpic}}\right>.\end{equation}
We prove the case where a simple circle intersects diagrams at two points. We have
\begin{eqnarray*}
& &\left<\raisebox{-10 pt}{\begin{overpic}[bb= 0 0 257 92, width=70 pt]{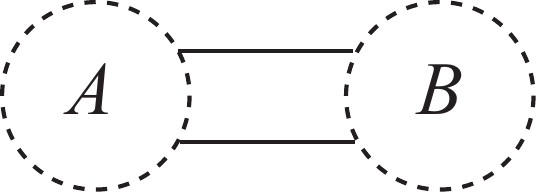}\put(45, 31){$\,_i$}\put(45, 3){$\,_j$}\end{overpic}}\right>\left<\raisebox{-10 pt}{\begin{overpic}[bb= 0 0 257 92, width=70 pt]{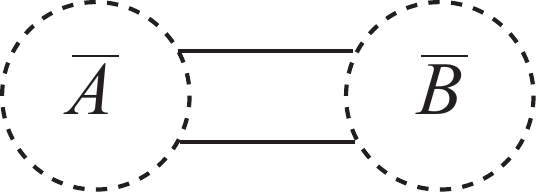}\put(45, 31){$\,_i$}\put(45, 3){$\,_j$}\end{overpic}}\right>\\
& &\hspace{-0.5cm}=\frac{1}{N^2}\left<\raisebox{-10 pt}{\begin{overpic}[bb= 0 0 257 92, width=70 pt]{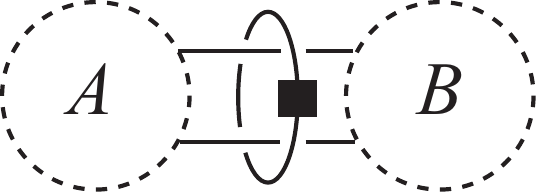}\put(34, 31){$\,_i$}\put(34, 3){$\,_j$}\end{overpic}}\right>\left<\raisebox{-10 pt}{\begin{overpic}[bb= 0 0 257 92, width=70 pt]{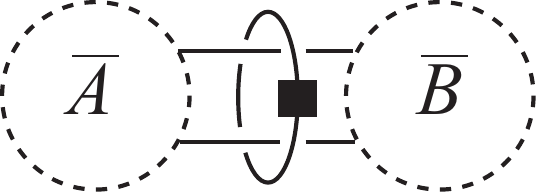}\put(34, 31){$\,_i$}\put(34, 3){$\,_j$}\end{overpic}}\right>\\
& &\hspace{-0.7cm}\overset{(\ref{eq11})}{=}\frac{1}{N^2}\frac{\delta_{ij}}{\Delta_i^2}\left<\raisebox{-10 pt}{\begin{overpic}[bb= 0 0 65 53, width=32 pt]{Lemma03.pdf}\put(77, 70){$\,_i$}\end{overpic}}\,\raisebox{-8 pt}{\includegraphics[bb= 0 0 84 196, width=10 pt]{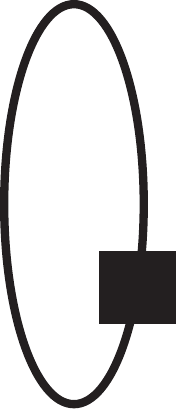}}\,\raisebox{-10 pt}{\begin{overpic}[bb=  0 0 65 53, width=32 pt]{Lemma04.pdf}\put(0, 70){$\,_i$}\end{overpic}}\right>\left<\raisebox{-10 pt}{\begin{overpic}[bb= 0 0 113 92, width=32 pt]{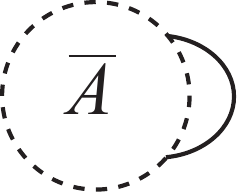}\put(77, 70){$\,_i$}\end{overpic}}\,\raisebox{-8 pt}{\includegraphics[bb= 0 0 84 196, width=10 pt]{Lemma14.pdf}}\,\raisebox{-10 pt}{\begin{overpic}[bb= 0 0 112 92, width=32 pt]{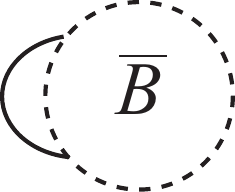}\put(0, 70){$\,_i$}\end{overpic}}\right>\\
& &\hspace{-0.6cm}\,\,=\,\frac{1}{N^2}\frac{\delta_{ij}}{\Delta_i^2}\left<\raisebox{-10 pt}{\begin{overpic}[bb= 0 0 65 53, width=32 pt]{Lemma03.pdf}\put(77, 70){$\,_i$}\end{overpic}}\,\raisebox{-8 pt}{\includegraphics[bb= 0 0 84 196, width=10 pt]{Lemma14.pdf}}\,\raisebox{-10 pt}{\begin{overpic}[bb= 0 0 112 92, width=32 pt]{Lemma04-2.pdf}\put(0, 70){$\,_i$}\end{overpic}}\right>\left<\raisebox{-10 pt}{\begin{overpic}[bb= 0 0 113 92,width=32 pt]{Lemma03-2.pdf}\put(77, 70){$\,_i$}\end{overpic}}\,\raisebox{-8 pt}{\includegraphics[bb= 0 0 84 196, width=10 pt]{Lemma14.pdf}}\,\raisebox{-10 pt}{\begin{overpic}[bb= 0 0 65 53, width=32 pt]{Lemma04.pdf}\put(0, 70){$\,_i$}\end{overpic}}\right>\\
& &\hspace{-0.7cm}\overset{(\ref{eq11})}{=}\frac{1}{N^2}\left<\raisebox{-10 pt}{\begin{overpic}[bb= 0 0 257 92, width=70 pt]{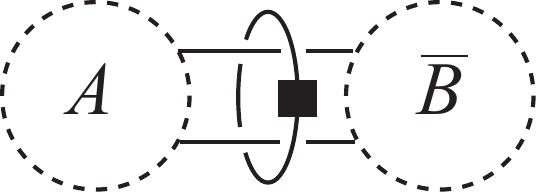}\put(34, 31){$\,_i$}\put(34, 3){$\,_j$}\end{overpic}}\right>\left<\raisebox{-10 pt}{\begin{overpic}[bb= 0 0 257 92, width=70 pt]{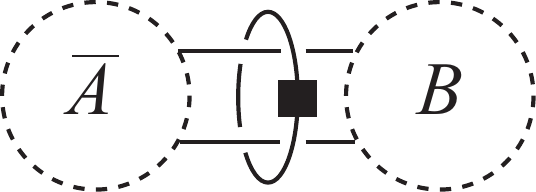}\put(34, 31){$\,_i$}\put(34, 3){$\,_j$}\end{overpic}}\right>\\
& &\hspace{-0.5cm}=\left<\raisebox{-10 pt}{\begin{overpic}[bb= 0 0 257 92, width=70 pt]{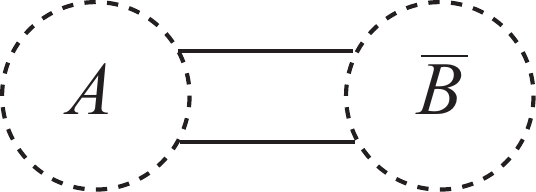}\put(45, 32){$\,_i$}\put(45, 3){$\,_j$}\end{overpic}}\right>\left<\raisebox{-10 pt}{\begin{overpic}[bb= 0 0 257 92,width=70 pt]{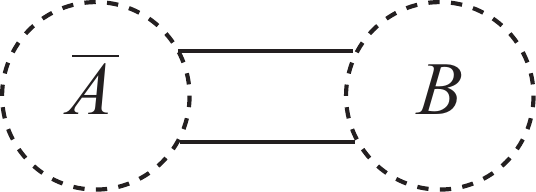}\put(45, 32){$\,_i$}\put(45, 3){$\,_j$}\end{overpic}}\right>,
\end{eqnarray*}
where the third equation holds because of the relation of the Kauffman bracket $\left<D\sqcup D'\right>=\left<D\right>\left<D'\right>$. This equation implies the required equation. The other cases are proved similarly. In the 0 point case, diagrams are disjoint. The 1 point case is by using Proposition \ref{Prop1}. The 3 point case is by using Eq.~(\ref{eq12}).
\end{proof}

\begin{prop} \label{Prop3} 
 Let $J$ be a handlebody-link and $\Gamma$ be a spatial trivalent graph that represents $J$. If $\Gamma$ has a loop edge that does not knot, we have another trivalent graph $\Gamma'$ by adding a full twist to the strands that pass through the loop edge. Let $J'$ be a handlebody-link that $\Gamma'$ represents, then $\left< J \right>_H = \left< J' \right>_H$.
\end{prop}
\begin{proof}This proposition is a corollary of Theorem \ref{Thm3} because $J$ and $J'$ have the homeomorphic exteriors.
\end{proof}
\par
By some changes of diagrams, we see that the pairs $\{ 5_3, 6_2 \}$, $\{ 6_5, 6_6 \}$ and $\{ 6_{14}, 6_{15} \}$ are related by the relation in Proposition \ref{Prop2} and the pairs  $\{ 5_1, 6_4\}$ and $\{5_2, 6_{13} \}$ are related by the relation in Proposition \ref{Prop3}. But the difference of pairs between $\{ 5_1, 6_4 \}$ and $\{5_2, 6_{13} \}$ has not been related by the relations above. 

\section{The Identification of $\left< \, \cdot \, \right>_H$ with the WRT Invariants} \label{sec5}
\par
In this section, we show that our invariants $\left< \, \cdot \, \right>_H$ for handlebody-links coincide with special cases of the WRT invariants \cite{RT} (see also \cite{BGM, Lic}) for closed oriented 3-manifolds. 
\begin{thm} \label{Thm3}
Let $J$ be a handlebody-link and  $E(J)$ be the exterior of $J$ in $S^3$. We take the double $B_{E(J)}$ of the exterior i.e. the space made by gluing boundaries of two copies of $E(J)$ by the natural orientation reversing homeomorphism. Then the following equation holds:
$$\left< J \right>_H = N^{\frac{g-l}{2}+1}Z_{\rm WRT}(B_{E(J)}),$$
where $Z_{\rm WRT}( \, \cdot \,)$ denotes the WRT invariants, $N$ is the value of the Kauffman bracket of the trivial circle with $\Omega$, $g$ is the sum of the genera of the components of $J$ and $l$ is the number of components of $J$.
\end{thm}
\par
In the rest of this section, we prove this theorem.
\subsection{The WRT invariants}
\par
We review a definition of the WRT invariants for closed oriented 3-manifolds. The WRT invariants can be calculated via \textit{Kirby diagrams}. Let $M$ be a closed oriented 3-manifold. A Kirby diagram $L_M$ of $M$ is a diagram of a framed link in $S^3$ on which the surgery makes $M$. In this paper, we use the blackboard framing to express framings of Kirby diagrams. It is well known that two Kirby diagrams representing the same closed manifold are transformed from one to another by a sequence of four moves RII, RIII, KI and KII. The moves RII and RIII are the second and third Reidemeister moves in Sec. \ref{sec2}. The KI move is adding a disjoint $\pm 1$ framed trivial component to diagrams or its inverse move i.e. removing a $\pm 1$ framed trivial component that is disjoint from the other components. The KII move is a band-sum of one component of a diagram and a parallel copy of another component of the diagram.

\begin{defn}[WRT invariants]
 Let $M$ be a closed oriented 3-manifold and $L_M$ be its Kirby diagram. Fix an integer $r \geq 3$ and we use the notations of Sec. \ref{sec4}. Let $\Omega L_M$ be the diagram derived by inserting $\Omega$ in each component of $L_M$ along its framing. The WRT invariants $Z_{\rm WRT}( \, \cdot \,)$ are defined by the following equation:
$$Z_{\rm WRT}(M) := N^{-\frac{t+1}{2}}\kappa^{-\sigma(L_M)}\left< \Omega L_M \right>,$$
where $t$ is the number of the components of $L_M$, $N=\left< \Omega \bigcirc \right>$, $\kappa = \exp (i\pi(r-2)(3-2r)/4r)$ and $\sigma(L_M)$ is the signature of the linking matrix of $L_M$ which is the symmetric  matrix whose $(i,j)$ entry $(i\neq j)$ is the linking number of the $i$th and $j$th components and the $(i,i)$ entry is the framing of the $i$th component. This value is invariant under the four moves RII, RIII, KI and KII.
\end{defn} 

\subsection{Proof of Theorem \ref{Thm3}}
\par
Let $J$ be a handlebody-link and $\Gamma$ be a spatial graph which represents $J$ such that each component of $\Gamma$ is a bouquet graph (i.e. a graph with just one vertex) or a circle. By adding new edges as Eq.~(\ref{eq1}), we expand the vertices of $\Gamma$ and have a spatial trivalent graph $\Gamma'$. Using Eqs.~(\ref{eq7}) and (\ref{eq10}), we have the following relation for the Kauffman bracket.
\[
\left<\,\,\,\,\, \raisebox{-17 pt}{\begin{overpic}[bb= 0 0 149 121, width=50 pt]{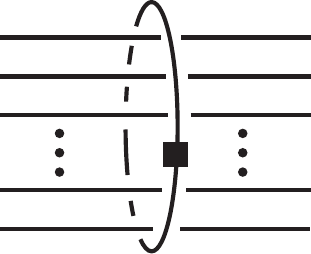}\put(-17, 70){$\,_{i_1}$}\put(-17, 57){$\,_{i_2}$}\put(-11, 25){$\vdots$}\put(-17, 7){$\,_{i_k}$}\end{overpic}}\, \right>=
\hspace{-0.3cm}\sum_{\substack{ \\x_1, x_2, \dots, x_l}} \hspace{-0.3cm} \frac{\Delta_{x_1}\Delta_{x_2}\cdots\Delta_{x_l}}{\theta(i_1, i_2, x_1) \cdots \theta(x_l, i_{k-1}, i_k)}
\left< \,\,\,\,\, \raisebox{-20 pt}{\begin{overpic}[bb= 0 0 66 107, width=28 pt]{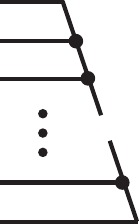}\put(-18, 100){$\,_{i_1}$}\put(-18, 80){$\,_{i_2}$}\put(-13, 32){$\vdots$}\put(-18, 2){$\,_{i_k}$}\put(40, 76){$\,_{x_1}$}\put(45, 58){$\,_{x_2}$}\put(54, 30){$\,_{x_l}$}\end{overpic}}\hspace{0.2cm}\raisebox{-11 pt}{\includegraphics[bb= 0 0 84 196, width=12 pt]{Lemma14.pdf}}\hspace{0.2cm}\raisebox{-20 pt}{\begin{overpic}[bb= 0 0 66 107, width=28 pt]{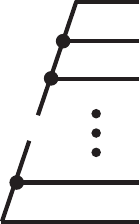}\put(62, 100){$\,_{i_1}$}\put(62, 80){$\,_{i_2}$}\put(66, 32){$\vdots$}\put(62, 2){$\,_{i_k}$}\put(2, 76){$\,_{x_1}$}\put(-5, 58){$\,_{x_2}$}\put(-13, 30){$\,_{x_l}$}\end{overpic}} \,\,\,\,\, \right>,
\]
where $4\leq k$ and $l=k-3$. Let $D_{\Gamma'}$ be a diagram of $\Gamma'$. The product of the Kauffman brackets of $D_{\Gamma'}$ and $\overline{D_{\Gamma'}}$ with some colors in the definition of $\left< J \right>_H$ is equal to the Kauffman bracket of the disjoint union of $D_{\Gamma'}$ and $\overline{D_{\Gamma'}}$: 
$$\left<D_{\Gamma'} (i_1, \dots, i_n)\right> \left<\overline{D_{\Gamma'}} (i_1, \dots, i_n)\right>= \left<D_{\Gamma'} (i_1, \dots, i_n)\sqcup\overline{D_{\Gamma'}} (i_1, \dots, i_n)\right>.$$
We add new $l$ trivial $\Omega$ circles to the diagram at a cost of $N^{-l}$ where $l$ is the number of components of $J$. Using the above relation with the $\Omega$ circles at expanded vertices of $\left<D_{\Gamma'} (i_1, \dots, i_n)\sqcup\overline{D_{\Gamma'}} (i_1, \dots, i_n)\right>$ and Eq.~(\ref{eq11}) to circle components, it changes to the Kauffman bracket of a diagram of the symmetric framed link $\left<S_{\Gamma} (j_1, \dots, j_m)\right>$ with remaining colors $(j_1, \dots, j_m)$ and $\Omega$ in ``vertical'' circles. $S_{\Gamma}$ consists of only circles (Fig.~\ref{fig7}). Then
$$\left< J \right>_H = N^{-l} \sum_{j_1, \dots , j_m} \prod_{k=1}^{m}\Delta_{j_k}\left<S_{\Gamma} (j_1, \dots, j_m)\right> = N^{-l} \left<\Omega S_{\Gamma}\right>,$$
where the second equation comes from the definition of $\Omega$. Using Eq.~(\ref{eq13}), one of the ``vertical" $\Omega$ circles can be taken away by sliding over other ``vertical" $\Omega$ circles. Therefore, $N^{-l} \left<\Omega S_{\Gamma}\right>= N^{-(l-1)}\left<\Omega S'_{\Gamma}\right>,$ where $S_{\Gamma}$ changed to $S'_{\Gamma}$ by taking away a ``vertical" circle (Fig.~\ref{fig7}). Since the diagram $S'_{\Gamma}$ is symmetric, the signature $\sigma(S'_{\Gamma})$ of the linking matrix of $S'_{\Gamma}$ is 0. The number $m$ is equal to the sum $g$ of the genera of the components of $J$. Then the number of the components of $S'_{\Gamma}$ is equal to $g+l-1$. We assume that $S'_{\Gamma}$ is a Kirby diagram for $B_{E(J)}$. Then the difference of the exponents of $N$ between $\left< J \right>_H$ and $Z_{\rm WRT}(B_{E(J)})$ is $-(l-1)-\left(-\frac{g+l-1+1}{2}\right) = \frac{g-l}{2}+1$.
\begin{figure}[ht]
 $$\raisebox{-30pt}{\includegraphics[bb= 0 0 196 147,height=65 pt]{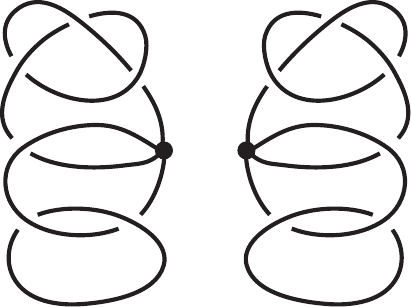}} \longrightarrow \raisebox{-30pt}{\includegraphics[bb= 0 0 196 147, height=65 pt]{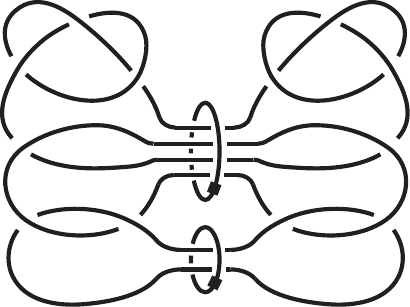}} \longrightarrow \raisebox{-30pt}{\includegraphics[bb= 0 0 196 147, height=65 pt]{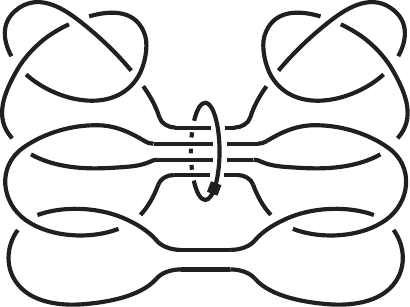}} $$
 \caption{$\Gamma$, $S_{\Gamma}$ and $S'_{\Gamma}$.} \label{fig7}
 \end{figure}
\par
We finally check that the diagram $S'_{\Gamma}$ is a Kirby diagram for the double space $B_{E(J)}$ of the exterior of  the handlebody-link $J$ in $S^3$. First, we check the case where $J$ is a handlebody-knot. We consider the exterior $E(J) (= E(\Gamma))$ as a 3-ball $B^3$ which is the complement in $S^3$ of the neighborhood of the vertex of $\Gamma$ with tunnels along the edges of $\Gamma$. The exterior $E(\overline{J})(=E(\overline{\Gamma}))$ of the mirror image of $J$ is also regarded in the same way. $S^3(=\mathbb{R}^3\cup\{\infty\})$ is divided into two 3-balls: $S^3= \{x\leq0\} \cup \{0\leq x\}$, where we use a coordinate of $\mathbb{R}^3$. We identify the 3-ball of $E(J)$ with $\{x\leq0\}$ and the 3-ball of $E(\overline{J})$ with $\{0\leq x\}$ and then glue the surfaces of the 3-balls naturally so that the ends of the corresponding tunnels of $E(J)$ and $E(\overline{J})$ are identified and the result is the exterior space $E(S'_{\Gamma})$ of the framed link that $S'_{\Gamma}$ represents (Fig.~\ref{fig8}). 
\begin{figure}[ht]
 $$\raisebox{-29pt}{\includegraphics[bb= 0 0 83 113, height=65 pt]{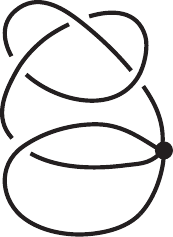}} \longrightarrow \raisebox{-40pt}{\includegraphics[bb= 0 0 127 191, height=90 pt]{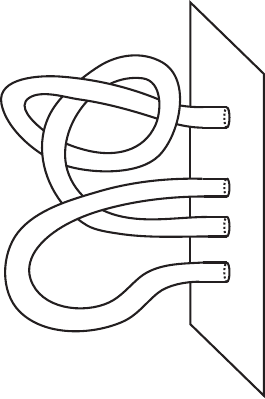}} \longrightarrow \raisebox{-40pt}{\includegraphics[bb= 0 0 218 191, height=90 pt]{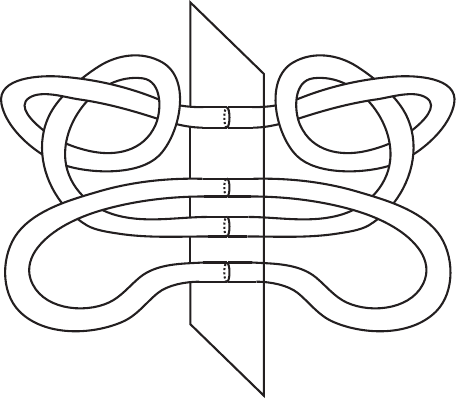}}$$
 \caption{$\Gamma$, $E(\Gamma)$ and $E(S'_{\Gamma})$.} \label{fig8}
 \end{figure}
In this case, $S'_{\Gamma}$ has no ``vertical'' circle. The space made from $E(S'_{\Gamma})$ by gluing corresponding points on the tunnel surfaces of $E(J)$ and $E(\overline{J})$ is the double space $B_{E(J)}$ of the exterior of $J$. The surgery on the framed link that $S'_{\Gamma}$ represents is gluing $m$ $D^2\times S^1$'s to the surfaces of $E(S'_{\Gamma})$ so that the boundary of each disc $D^2$ of $D^2\times S^1$ coincides with a longitude line of the surface along the framing. In this case, the framing of each component is 0 and gluing the discs to the surfaces of $E(S'_{\Gamma})$ is equal to gluing the corresponding surface points of the tunnels of $E(J)$ and $E(\overline{J})$. Therefore, the result of the surgery is the double space $B_{E(J)}$. 
\par
Next, we check the case where $J$ is a handlebody-link whose number of components is more than 1. We see the surgery on a ``vertical'' circle. Assume that the ``vertical'' circle is in the plane $\{x=0\}$ in $S^3$. The surgery on the circle is removing the neighborhood of the circle from $S^3$ and refilling the hole with $D^2\times S^1$ so that the boundary of each disc $D^2$ coincides with the 0 framing longitude of the hole. We divide $S^3$ and $D^2\times S^1$ into two parts: $S^3= \{x\leq0\} \cup \{0\leq x\}$ and $D^2\times S^1 = D^2 \times I \cup D^2 \times I$, where $I$ is the unit interval. We can see the surgery as the operation first gluing the two $D^2 \times I$'s to $\{x\leq0\}$ and $\{0\leq x\}$ along annuli where the ``vertical'' circle was removed and then gluing $\{x\leq0\} \cup D^2 \times I$ and $D^2 \times I \cup \{0\leq x\}$ as in Fig.~\ref{fig9}, where in the fourth figure an inside surface is glued to another inside surface and an outside surface is glued to another outside surface.
\begin{figure}[ht]
$$\raisebox{-37pt}{\includegraphics[bb= 0 0 35 189, height=80 pt]{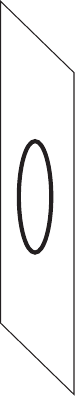}}\longrightarrow \raisebox{-37pt}{\includegraphics[bb= 0 0 35 189, height=80 pt]{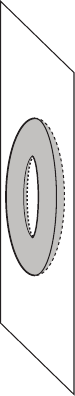}}\longrightarrow \raisebox{-37pt}{\includegraphics[bb= 0 0 35 189, height=80 pt]{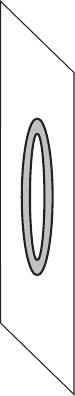}} \,\, \raisebox{-13pt}{\includegraphics[bb= 0 0 46 76, height=31 pt]{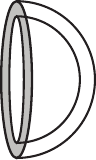}} \,\, \raisebox{-13pt}{\includegraphics[bb= 0 0 46 76, height=31 pt]{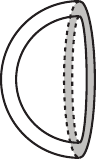}} \,\, \raisebox{-37pt}{\includegraphics[bb=  0 0 35 189, height=80 pt]{surgery02-5.pdf}}\longrightarrow \raisebox{-37pt}{\includegraphics[bb=  0 0 35 189, height=80 pt]{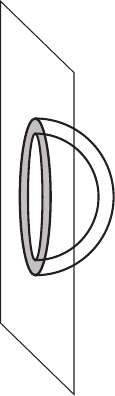}} \,\,\,\,\,\,\,\mbox{\Large{$\cup$}} \,\, \raisebox{-37pt}{\includegraphics[bb= 0 0 56 189, height=80 pt]{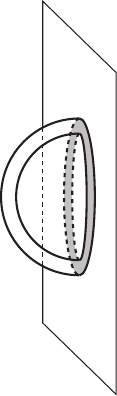}}$$
\caption{The surgery on a ``vertical'' circle.} \label{fig9}
\end{figure}
Then we consider the surgery on $S'_{\Gamma}$. We place the framed link that $S'_{\Gamma}$ represents in a symmetric position with respect to the plane $\{x=0\}$ in $S^3$ such that the ``vertical'' circles are in the plane $\{x=0\}$. Let $E(S'_{\Gamma})_{-}$ be the intersection $E(S'_{\Gamma}) \cap \{x\leq0\}$ and $E(S'_{\Gamma})_{+}$ be $E(S'_{\Gamma})\cap\{0\leq x\}$. We do the surgery by attaching some $D^2 \times I$'s to $E(S'_{\Gamma})_{-}$ and $E(S'_{\Gamma})_{+}$ along annuli where the ``vertical'' circles were removed and then gluing the corresponding surfaces of $E(S'_{\Gamma})_{-}\cup \{D^2 \times I$'s$\}$ and $E(S'_{\Gamma})_{+}\cup \{D^2 \times I$'s$\}$ (Fig.~\ref{fig10}). The spaces $E(S'_{\Gamma})_{-} \cup \{D^2 \times I$'s$\}$ and $E(S'_{\Gamma})_{+} \cup \{D^2 \times I$'s$\}$ are equal to $E(J)$ and $E(\overline{J})$ respectively (Fig.~\ref{fig10} (right)). Therefore, the result of the surgery on $S'_{\Gamma}$ is the double space $B_{E(J)}$ of the exterior of the handlebody-link $J$. This completes the proof of Theorem \ref{Thm3}.

\begin{figure}[ht]
$$\raisebox{-33pt}{\includegraphics[bb= 0 0 196 147, height=70 pt]{graph18.pdf}}\,\, \longrightarrow \,\, \raisebox{-45pt}{\includegraphics[bb= 0 0 173 240, height=100 pt]{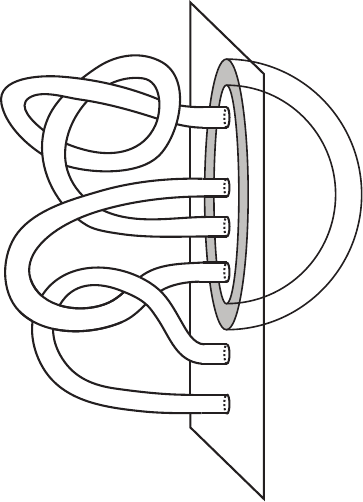}} \,\,\, \mbox{\Large{$\cup$}} \,\,\, \raisebox{-45pt}{\includegraphics[bb= 0 0 173 240, height=100 pt]{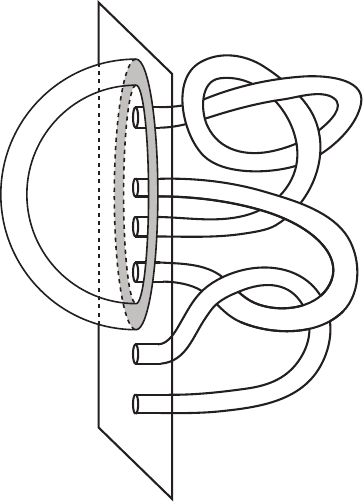}}$$
\caption{The surgery on $S'_{\Gamma}$.} \label{fig10}
\end{figure}

\section*{Acknowledgments}
The authors would like to thank Atsushi Ishii for his helpful comments. The second author was partly supported by JSPS KAKENHI Grant Number 22540236, 23340115.

\end{document}